\documentclass[11pt, a4paper]{amsart}
\usepackage{a4wide}
\usepackage[latin1]{inputenc}
\usepackage{amsmath}
\usepackage{amsfonts}
\usepackage{amssymb}
\usepackage{graphicx}
\usepackage{amsthm}
\usepackage{amssymb}
\usepackage{cite}
\usepackage{mathrsfs}
\usepackage{hyperref}
\usepackage{paralist}
\usepackage{enumitem}
\usepackage{tikz}
\usepackage[margin=1in]{geometry}
\usetikzlibrary{matrix,chains}
\usepackage{dynkin-diagrams}

\newtheorem{thml}{Theorem}

\newtheorem{corl}[thml]{Corollary}
\newtheorem{condl}{Condition}

\newtheorem*{MNConj}{McKay--Navarro Conjecture}

\usepackage{prettyref}
\newrefformat{sec}{Section \ref{#1}}
\newrefformat{chap}{Chapter \ref{#1}}
\newrefformat{prop}{Proposition \ref{#1}}
\newrefformat{prob}{Problem \ref{#1}}
\newrefformat{rem}{Remark \ref{#1}}
\newrefformat{cor}{Corollary \ref{#1}}
\newrefformat{cond}{Condition \ref{#1}}
\newrefformat{def}{Definition \ref{#1}}
\newrefformat{ex}{Example \ref{#1}}
\newrefformat{lem}{Lemma \ref{#1}}
\newrefformat{eq}{Equation \eqref{#1}}
\newrefformat{con}{condition (\ref{#1})}
\newrefformat{conj}{Conjecture \ref{#1}}
\newrefformat{fig}{Figure \ref{#1}}
\newrefformat{tab}{Table \ref{#1}}
\newrefformat{app}{Appendix \ref{#1}}

\newcommand{\aut}{\mathrm{Aut}}
\newcommand{\out}{\mathrm{Out}}

\newcommand{\irr}{\mathrm{Irr}}

\newcommand{\syl}{\mathrm{Syl}}
\newcommand{\wh}[1]{\widehat{#1}}

\newcommand{\la}{\lambda}

\newcommand{\wt}[1]{\widetilde{#1}}

\newcommand{\bg}[1]{\mathbf{#1}}
\newcommand{\norm}{\mathrm{N}}
\newcommand{\cen}{\mathrm{C}}
\newcommand{\cent}{\mathrm{C}}
\newcommand{\zen}{\mathrm{Z}}
\newcommand{\zent}{\mathrm{Z}}
\newcommand{\HC}{\mathrm{R}}
\newcommand{\type}[1]{\mathsf{#1}}
\newcommand{\Sp}{\operatorname{Sp}}
\newcommand{\SO}{\operatorname{SO}}
\newcommand{\gal}{\mathcal{G}}
\newcommand{\galh}{\mathcal{H}}
\newcommand{\F}{\mathbb{F}}
\newcommand{\Q}{\mathbb{Q}}
\newcommand{\alt}{\mathfrak{A}}
\newcommand{\sym}{\mathfrak{S}}

\newcommand{\odd}{\mathrm{Irr}_{2\mathrm{cen}}}

\begin{document}
	
	\sloppy
	
	\newcommand{\n}{\mathbf{n}}
	\newcommand{\x}{\mathbf{x}}
	\newcommand{\h}{\mathbf{h}}
	\newcommand{\m}{\mathbf{m}}
	\newcommand{\Ph}{\Phi_}
	
	\newcommand{\bL}{\mathbf{L}}
	\newcommand{\bT}{\mathbf{T}}
	\newcommand{\bG}{\mathbf{G}}
	\newcommand{\bS}{\mathbf{S}}
	\newcommand{\bZ}{\mathbf{Z}}
	\newcommand{\bH}{\mathbf{H}}
	\newcommand{\B}{\mathbf{B}}
	\newcommand{\U}{\mathbf{U}}
	\newcommand{\V}{\mathbf{V}}
	\newcommand{\T}{\mathbf{T}}
	\newcommand{\G}{\mathbf{G}}
	\newcommand{\Para}{\mathbf{P}}
	\newcommand{\Levi}{\mathbf{L}}
	\newcommand{\Y}{\mathbf{Y}}
	\newcommand{\X}{\mathbf{X}}
	\newcommand{\M}{\mathbf{M}}
	\newcommand{\pro}{\mathbf{prod}}
	\renewcommand{\o}{\overline}
	
	\newcommand{\Gtilde}{\mathbf{\tilde{G}}}
	\newcommand{\Ttilde}{\mathbf{\tilde{T}}}
	\newcommand{\Btilde}{\mathbf{\tilde{B}}}
	\newcommand{\Ltilde}{\mathbf{\tilde{L}}}
	\newcommand{\C}{\operatorname{C}}
	
	\newcommand{\N}{\operatorname{N}}
	\newcommand{\bl}{\operatorname{bl}}
	\newcommand{\Z}{\operatorname{Z}}
	\newcommand{\Gal}{\operatorname{Gal}}
	\newcommand{\modulo}{\operatorname{mod}}
	\newcommand{\kernel}{\operatorname{ker}}
	\newcommand{\Irr}{\operatorname{Irr}}
	\newcommand{\D}{\operatorname{D}}
	\newcommand{\I}{\operatorname{I}}
	\newcommand{\GL}{\operatorname{GL}}
\newcommand{\PSU}{\operatorname{PSU}}
		
	\newcommand{\SL}{\operatorname{SL}}
	\newcommand{\W}{\operatorname{W}}
	\newcommand{\R}{\operatorname{R}}
	\newcommand{\Br}{\operatorname{Br}}
	\newcommand{\Aut}{\operatorname{Aut}}
	\newcommand{\End}{\operatorname{End}}
	\newcommand{\Ind}{\operatorname{Ind}}
	\newcommand{\Res}{\operatorname{Res}}
	\newcommand{\br}{\operatorname{br}}
	\newcommand{\Hom}{\operatorname{Hom}}
	\newcommand{\Endo}{\operatorname{End}}
	\newcommand{\Ho}{\operatorname{H}}
	\newcommand{\Tr}{\operatorname{Tr}}
	\newcommand{\opp}{\operatorname{opp}}
	\newcommand{\ssc}{\operatorname{sc}}
	\newcommand{\ad}{\operatorname{ad}}

	\newcommand{\tw}[1]{{}^#1\!}
	
	\theoremstyle{definition}
	\newtheorem{definition}{Definition}[section]
	\newtheorem{notation}[definition]{Notation}
	\newtheorem{construction}[definition]{Construction}
	\newtheorem{remark}[definition]{Remark}
	\newtheorem{example}[definition]{Example}

	\theoremstyle{theorem}
	\newtheorem{theorem}[definition]{Theorem}
	\newtheorem{lemma}[definition]{Lemma}
	\newtheorem{question}{Question}
	\newtheorem{corollary}[definition]{Corollary}
	\newtheorem{proposition}[definition]{Proposition}
	\newtheorem{conjecture}[definition]{Conjecture}
	\newtheorem{assumption}[definition]{Assumption}
	\newtheorem{hypothesis}[definition]{Hypothesis}
	\newtheorem{maintheorem}[definition]{Main Theorem}

	\newtheorem{theo}{Theorem}
	\newtheorem{conj}[theo]{Conjecture}
	\newtheorem{cor}[theo]{Corollary}

	\renewcommand{\thetheo}{\Alph{theo}}
	\renewcommand{\theconj}{\Alph{conj}}
	\renewcommand{\thecor}{\Alph{cor}}
	
	\newcommand{\mandicomment}{\textcolor{blue}}
	\newcommand{\lucascomment}{\textcolor{purple}}
	
\title[McKay--Navarro]{The McKay--Navarro conjecture for the prime 2}
\author{L. Ruhstorfer}
\address[L. Ruhstorfer]{{School of Mathematics and Natural Sciences}, {University of Wuppertal}, {Gaußstr. 20,
		42119 Wuppertal, Germany}}
\email{ruhstorfer@uni-wuppertal.de}

\author{A. A. Schaeffer Fry}
\address[A. A. Schaeffer Fry]{Dept. Mathematics - University of Denver, Denver, CO 80204, USA}
\email{mandi.schaefferfry@du.edu}

\date{}
\thanks{The authors thank the  Isaac Newton Institute
for Mathematical Sciences in Cambridge and the organizers of the Summer 2022 INI program Groups, Representations, and
Applications: New Perspectives, supported by EPSRC grant EP/R014604/1, where this work began.  The second author also gratefully acknowledges support from the National Science Foundation, Award No. DMS-2100912 and her previous institution MSU Denver, who allowed her to continue serving as PI.  The authors thank Carolina Vallejo for useful conversations on the topic for sporadics and alternating groups and Gunter Malle and Britta Sp{\"a}th for their comments on an early draft. Finally, they thank the anonymous referees, whose careful reading and thoughtful remarks greatly improved the exposition.}
	
	\maketitle
\begin{abstract}
We complete the proof of the McKay--Navarro conjecture (also known as the Galois--McKay conjecture) for the prime 2, by completing the proof of the inductive McKay--Navarro conditions introduced by Navarro--Sp{\"a}th--Vallejo for this prime.

\vspace{0.25cm}

\noindent \textit{Mathematics Classification Number:} 20C15, 20C33


\end{abstract}

In the last several decades, much of the progress in the character theory of finite groups has been driven by the so-called ``local-global" conjectures, which posit various relationships between the character theory of a finite group $G$ and the character theory of (or properties of) certain ``local" subgroups, such as the normalizer $\norm_G(Q)$ of a Sylow $\ell$-subgroup $Q$ for a prime $\ell$. One of the most famous of these conjectures is the McKay conjecture, which suggests that there should exist a bijection $\irr_{\ell'}(G)\leftrightarrow\irr_{\ell'}(\norm_G(Q))$ between the set of irreducible complex characters of $G$ whose degrees are not divisible by the prime $\ell$ and the corresponding set for $\norm_G(Q)$. Despite its seemingly innocent formulation, this conjecture has proven to be extremely elusive, with experts still seeking a more conceptual reason for its validity.  In a landmark paper by Isaacs--Malle--Navarro \cite{IsaacsMalleNavarroMcKayreduction}, the McKay conjecture was reduced to a set of stronger conditions on simple groups, now often known as the ``inductive McKay" conditions. With this in place, Malle--Sp{\"a}th and Sp{\"a}th \cite{MS16, spath23} were able to complete the proof of the McKay conjecture for the primes $2$ and $3$, respectively, by proving the inductive conditions for simple groups. While the current paper was under review, the full McKay conjecture was even announced complete by Cabanes--Sp{\"a}th \cite{CS24} using this path.

 In 2004, Navarro noted that there should exist a bijection as predicted in the McKay conjecture that holds a much stronger property, relating not just the character degrees but also the other character values.  To be more precise, let $\gal$ denote the Galois group $\mathrm{Gal}(\Q(e^{2\pi i/|G|})/\Q)$, which acts naturally on the set $\irr(G)$ of irreducible complex characters of $G$ via $\chi^\sigma(g):=\sigma(\chi(g))$ for each $g\in G$ and $\sigma\in\gal$, and let $\galh_\ell\leq \gal$ denote the subgroup consisting of those $\sigma\in\gal$ for which there is some nonnegative integer $e$ such that $\sigma(\zeta)=\zeta^{\ell^e}$ for every root of unity $\zeta$ of order not divisible by $\ell$. Then Navarro's refinement of the McKay conjecture \cite{navarro2004}, which is often known as the McKay--Navarro conjecture or the Galois--McKay conjecture, is as follows:
 
 \begin{MNConj}[Navarro 2004]
 Let $G$ be a finite group, let $\ell$ be a prime, and let $Q\in\mathrm{Syl}_\ell(G)$.  Then there exists an $\galh_\ell$-equivariant bijection $\irr_{\ell'}(G)\leftrightarrow\irr_{\ell'}(\norm_G(Q))$.
  \end{MNConj}

This conjecture has recently inspired a large amount of study on the action of Galois automorphisms on $\irr(G)$ for various groups and the fields of values of characters, see e.g. \cite{SrinivasanVinroot, SFT22, SFV19}.  It has also led to the proof of many interesting statements about properties of $G$ that one can determine from the character table, which would have followed from the McKay--Navarro conjecture or its blockwise version, see e.g. \cite{NTT07, SFgaloisHC, NV17, NTV, SFTaylor2, NT19, malle19, RSV20, NRSV, hung24}.

In \cite{NavarroSpathVallejo}, Navarro--Sp{\"a}th--Vallejo reduced the McKay--Navarro conjecture to proving the so-called \emph{inductive Galois--McKay}, which we will call the  \emph{inductive McKay--Navarro} conditions, for the universal covering groups of nonabelian simple groups, building upon the strategy from Isaacs--Malle--Navarro's inductive McKay conditions. These conditions can be roughly broken into an ``equivariant" condition and an ``extension" condition. 

The first-named author proved that these conditions hold for groups of Lie type in defining characteristic \cite{ruhstorfer}, with exceptions for finitely many types dealt with by Johansson \cite{johansson22} (including the Tits group and groups with exceptional Schur multipliers).  With this, it now suffices, for groups of Lie type, to prove the inductive conditions when the prime $\ell$ is not the defining characteristic for $G$.

In \cite{SF22}, the second-named author proved the ``equivariant" part of the conditions for most groups of Lie type defined over $\mathbb{F}_q$ for the prime 2 and other primes dividing $q-1$, and the two authors proved in  \cite{RSF22} that the full inductive McKay--Navarro conditions hold for untwisted groups of Lie type without graph automorphisms for the prime $\ell=2$. For $\ell=2$ and groups of Lie type defined in odd characteristic, this leaves the case of groups whose underlying algebraic group admits graph automorphisms. 
Some of the latter groups, namely the Suzuki and Ree groups, $\type{G}_2(3^a)$, $\tw{3}\type{D}_4(q)$, and $\tw{2}\type{E}_6(q)$, have also been completed by Johansson \cite{birtethesis}.  

Here, we complete the picture with the following:

\begin{thml}\label{thm:main}
Let $q$ be a power of an odd prime. Then the inductive McKay--Navarro conditions hold for the prime $\ell=2$ for the following simple groups:
\begin{enumerate}
\item\label{mainspor} sporadic groups and groups of Lie type defined over $\F_q$ with exceptional Schur multipliers; 
\item\label{mainalt} the alternating groups $\alt_n$ with $n\geq 5$;
\item\label{mainlie}  the groups of Lie type $\type{D}_n(q)$ and $\tw{2}\type{D}_n(q)$ for $n\geq 4$;  $\type{A}_n(q)$ and $\tw{2}\type{A}_n(q)$ for $n\geq 2$; and $\type{E}_6(q)$ and $\tw{2}\type{E}_6(q)$. 
\end{enumerate} 
\end{thml}

\begin{corl}\label{cor:MN2}
The McKay--Navarro conjecture holds for the prime $\ell=2$.
\end{corl}

We make a few remarks.  We first note that, thanks to the results mentioned above,   Corollary \ref{cor:MN2} follows from \prettyref{thm:main}. We also remark that we include $\tw{2}\type{E}_6(q)$ (also settled by \cite{birtethesis}), since it can be addressed naturally alongside the others listed. The inductive McKay--Navarro conditions for $\ell=2$ for sporadic and alternating groups were originally checked by C. Vallejo (private communication), and we discuss them briefly in Section \ref{sec:sporalt}. We also note that Corollary \ref{cor:MN2} recovers the main results of \cite{hung24} and of \cite{SFgaloisHC}, although the work in the latter was essential for the results here.
	
The remainder of the paper is structured as follows. In Section \ref{sec:iMN}, we recall the inductive McKay--Navarro conditions. 
In Section \ref{sec:sporalt}, we prove Theorem \ref{thm:main} (\ref{mainspor}) and (\ref{mainalt}).   
In Section \ref{sec:prelim}, we introduce some notation to be used in the remaining sections, which are devoted to the groups in Theorem \ref{thm:main}(\ref{mainlie}).  
We discuss odd-degree characters of these groups and their rationality in Section \ref{sec:rationality}, which is used in Section \ref{sec:equiv} to prove the ``equivariant" part of the inductive McKay--Navarro conditions.  
Finally, we prove the ``extension" part of the conditions in Section \ref{sec:extensions}, completing the proof of Theorem \ref{thm:main}.

\section{The Inductive McKay--Navarro Conditions}\label{sec:iMN}

Fix a prime $\ell$, and write $\galh:=\galh_\ell$. For a character $\chi$, let $\chi^\galh$ denote the orbit under $\galh$ of $\chi$.  When a group $A$ acts on a set $\mathcal{X}$, we will write $A_{\omega}$ for the stabilizer of the element $\omega\in\mathcal{X}$ under $A$.  So, for example, for a finite group $G$, $Q\leq G$, and $\chi\in\irr(G)$, we use $\aut(G)_Q$ to denote the subgroup of automorphisms normalizing $Q$, and $\aut(G)_{Q,\chi^\galh}$ is the stabilizer in $\aut(G)_Q$ of the orbit $\chi^\galh$. That is, $\aut(G)_{Q,\chi^\galh}$ is the group of $\alpha\in\aut(G)_Q$ such that $\chi^\alpha=\chi^\sigma$ for some $\sigma\in\galh$. When $G$ and $Q$ are understood, we will often use $\Gamma$ to denote the group $\Gamma:=\aut(G)_Q$, so that $\Gamma_{\chi^\galh}:=\aut(G)_{Q,\chi^\galh}$, as is frequently done in the existing literature.

 The inductive McKay--Navarro conditions \cite[Definition 3.1]{NavarroSpathVallejo} require two main parts.  The ``equivariant" condition can be phrased as follows: 

\begin{condl}\label{cond:condition}
Let $G$ be a finite quasisimple group, let $Q\in\syl_\ell(G)$, and write $\Gamma:=\aut(G)_Q$.  Then there is a proper $\Gamma$-stable subgroup $M$ of $G$ with $\norm_G(Q)\leqslant M$ and a $\Gamma\times\galh_\ell$-equivariant bijection 
$\Omega\colon\irr_{\ell'}(G)\rightarrow\irr_{\ell'}(M)$ such that corresponding characters lie over the same character of $\zen(G)$.
\end{condl}

For the inductive McKay--Navarro conditions to hold for a finite nonabelian simple group $S$, we require that its universal covering group $G$ satisfies Condition \ref{cond:condition} and that 
\begin{equation}\label{eq:extpart}
(G\rtimes \Gamma_{\chi^{\galh}}, G, \chi)_{\galh} \geqslant_c (M\rtimes \Gamma_{\chi^{\galh}}, M, \Omega(\chi))_{\galh}
\end{equation}
 for all $\chi \in \Irr_{\ell'}(G)$ in the notation of \cite[Definition 1.5]{NavarroSpathVallejo}, which we will now recall.

Let $G\lhd A$ be finite groups and let $\chi\in\Irr(A)$ be such that the orbit $\chi^\galh$ is stable under $A$. Then we write $(A, G, \chi)_\galh$ and call this an $\galh$-triple. Note then that $(A_\chi, G, \chi)$ is a character triple in the sense of \cite[pp. 186]{isaacs}.  As such, there exists a projective representation $\mathcal{P}$ of $A_\chi$ associated to $\chi$, by \cite[Theorem 11.2]{isaacs}. Let $\Q^{\mathrm{ab}}$ be the field generated by $\Q$ and all roots of unity in $\mathbb{C}$. By \cite[Corollary 1.2]{NavarroSpathVallejo}, $\mathcal{P}$ may be taken to have entries in $\Q^{\mathrm{ab}}$ and  such that the factor set $\alpha\colon G\times G\rightarrow \Q^{\mathrm{ab}}$ satisfying $\mathcal{P}(x)\mathcal{P}(y)=\mathcal{P}(xy)\alpha(x,y)$ for each $x,y\in G$ (see \cite[Definition 11.1]{isaacs}) takes only roots of unity as values.  Further, for any $a\in (A\times \galh)_\chi$, there is a function $\mu_a\colon A_\chi\rightarrow \mathbb{C}^\times$ such that $\mathcal{P}^a$ is similar to $\mu_a\mathcal{P}$, by \cite[Lemma 1.4]{NavarroSpathVallejo}.

 Finally, we are ready to define the relation mentioned in \eqref{eq:extpart}.

\begin{definition}[Definition 1.5 of \cite{NavarroSpathVallejo}]\label{def:triples}

Let $(A, G, \chi)_\galh$ and $(H, M, \psi)_\galh$ be $\galh$-triples. Then we write $(A, G, \chi)_\galh\geqslant_c (H, M, \psi)_\galh$ if the following hold:
\begin{enumerate}[label=(\roman*)]
\item $A=GH$, $\cent_A(G)\leq H$, and $M=H\cap G$.
\item $(H\times\galh)_\psi=(H\times\galh)_\chi$.
\item There are projective representations $\mathcal{P}$ of $A_\chi$ and $\mathcal{P}'$ of $H_\psi$ associated to $\chi$ and $\psi$ with entries in $\Q^{\mathrm{ab}}$ with factor sets $\alpha$ and $\alpha'$, respectively, such that $\alpha$ and $\alpha'$ take root of unity values and coincide on $H_\chi\times H_\chi$. Further, the scalar matrices $\mathcal{P}(c)$ and $\mathcal{P'}(c)$ correspond to the same scalar for each $c\in \cent_{A}(G)$.
\item For each $a\in(H\times\galh)_\chi$, the functions $\mu_a$ and $\mu_a'$ such that $\mathcal{P}^a$ is similar to $\mu_a\mathcal{P}$ and $\mathcal{P'}^a$ is similar to $\mu_a'\mathcal{P'}$ (see \cite[Lemma 1.4]{NavarroSpathVallejo}) agree on $H_\chi$.
\end{enumerate}
\end{definition}

We remark that the triples in \eqref{eq:extpart} automatically satisfy conditions (i) and (ii) (see \cite[Lemma 1.33]{birtethesis}), which reduces checking  \eqref{eq:extpart} to verifying the ``extension" conditions (iii) and (iv) of Definition \ref{def:triples} for the triples in \eqref{eq:extpart}.

Throughout, given a nonabelian simple group $S$, we will say that \emph{the inductive McKay--Navarro conditions hold for $S$} if Conditions \ref{cond:condition} 
and \eqref{eq:extpart} hold for its universal covering group $G$.  
We remark that this gives the conditions of \cite[Definition 3.1]{NavarroSpathVallejo}, originally called the inductive Galois--McKay conditions.
 By \cite[Theorem A]{NavarroSpathVallejo}, the McKay--Navarro conjecture holds for all finite groups for the prime $\ell$  if the inductive McKay--Navarro conditions hold for all finite nonabelian simple groups for the prime $\ell$.

In practice,  we are often able to replace the groups $G\rtimes \aut(G)_{Q,\chi^{\galh}}$ and $M\rtimes \aut(G)_{Q, \chi^{\galh}}$ with more convenient groups inducing the same set of automorphisms, using \cite[Theorem 2.9]{NavarroSpathVallejo}. This will be especially useful in our proof of Theorem \ref{thm:main}(\ref{mainlie}). 

\section{Alternating Groups, Sporadic Groups, and Groups with Exceptional Schur Multipliers}\label{sec:sporalt}

We collect here some results that can be completed more computationally or follow directly from previous results.

\begin{proposition}\label{prop:sporadicalt}
Theorem \ref{thm:main} holds for the groups listed in Theorem \ref{thm:main}(\ref{mainspor}) and (\ref{mainalt}). 
\end{proposition}
\begin{proof}
Let $S$ be one of the simple groups listed. We first note that it suffices to prove Conditions \ref{cond:condition} and \eqref{eq:extpart} from Section \ref{sec:iMN} for the $2'$-covering group of $S$, using e.g. \cite[Lemma 4.15]{birtethesis}. 

First let $S=\alt_n$ be an alternating group with $n\geq 8$, so that $\aut(S)=\sym_n$, both $S$ and $\aut(S)$ have self-normalizing Sylow $2$-subgroups, and the abelianization of these Sylow $2$-subgroups are elementary abelian (see e.g. \cite[Lemma 4.1]{NRSV}). This means that all odd-degree characters of normalizers of Sylow $2$-subgroups of $S$ and of $\aut(S)$ are $\galh_2$-invariant (and even rational-valued).  
The group $S$ is known to satisfy the inductive McKay conditions for the prime $2$ by \cite[Theorem 1.1]{malle08b} and to satisfy the McKay--Navarro conjecture by \cite{nath}. In particular, we then have that $S$ satisfies Condition \ref{cond:condition}, with $M:=\norm_G(Q)=Q$ for $Q\in\syl_\ell(S)$. In fact, every member of $\irr_{2'}(S)$ is fixed by $\galh_2$ by \cite[Theorem 2.3]{nath}. 
Further, it is well-known  that the characters of $\sym_n$ are rational-valued. Note that the stabilizer in $\sym_n$ of any $\chi\in\irr_{2'}(S)$ is either $S$ or $\sym_n$, and in the latter case, $\chi$ extends to $\irr_{2'}(\sym_n)$. Now, let $\hat Q\in\syl_2(\sym_n)$ contain $Q$, and note that $S\hat Q=\sym_n$. If $\Omega$ is the map from Condition \ref{cond:condition}, then the equivariance of $\Omega$ implies that $\Omega(\chi)$ extends to $\hat Q$ exactly when $\chi$ extends to $\sym_n$ and has stabilizer $Q$ in $\hat Q$ otherwise.
 Further, as noted before, these extensions of $\chi$ and of $\Omega(\chi)$, in the case that they exist, are $\galh_2$-invariant. Hence we see that $(\sym_n, S, \chi)_{\galh_2} \geqslant_c (\hat Q, Q, \Omega(\chi))_{\galh_2}$. But by applying \cite[Theorem 2.9]{NavarroSpathVallejo}, we therefore see that \eqref{eq:extpart} holds.

Johansson \cite{birtethesis} has completed the cases of the groups $\alt_6\cong \type{B}_2(2)'$, $\type{G}_2(3)$, and the  groups that can be identified with groups of Lie type with exceptional Schur multiplier in defining characteristic. This leaves the sporadic groups,  $ \alt_7$, $\operatorname{PSU}_4(3)$, and $\type{B}_3(3)$.  (We remark that the sporadic groups were also checked by C. Vallejo - private communication.) These groups can be dealt with using the character tables in GAP, building on the fact that they satisfy the inductive McKay conditions by \cite{malle08b}; the knowledge of the action of $\aut(S)$ on the Schur multiplier (see \cite[Table 6.3.1]{GLS3}); the fact that in many cases, the Sylow $2$-subgroup is self-normalizing; and arguing as in \cite[Proposition 5.13]{birtethesis}. We provide more details below. 

Let $S$ be one of the sporadic simple groups, $\alt_7$, or $\type{B}_3(3)$. In each case, let $G$ be the odd covering group, which is either $3.S$ or $S$ itself, and let $Z:=\zent(G)$. Further, $A=\Aut(S)$ is either $S.2$ or $S$ itself. Let $\hat A:=3.A$, when applicable, and otherwise define $\hat A:=A$. The character tables for $S, G, A$, and $\hat A$ are available in the GAP Character Table Library, and we can further compute with the underlying groups themselves using the GAP commands \verb+GroupForGroupInfo+  and \verb+GroupInfoForCharacterTable+.  Let $Q\in\syl_2(\hat A)$, $P:=P\cap Q\in \syl_2(G)$,  $\bar P:=PZ/Z\in\syl_2(S)$, and $\bar Q:=QZ/Z\in\syl_2(A)$. The structure of $\norm_S(\bar P)/\bar{P}'$ when $S$ is sporadic is given in \cite[Table 1]{wilson}, and this is elementary abelian unless $S\in\{J_1, J_2, J_3, Suz, HN\}$. 

Computing with this, and sometimes applying \cite[Theorem A]{NT16}, we obtain the following information, in all cases under consideration:
\begin{itemize}
\item $Q/Q'$ and $P/P'$ are elementary abelian;
\item $\norm_G(P)/P'\cong Z\times \norm_S(\bar P)/\bar P'$;
\item $\norm_{\hat A}( P)/ P'=\left(\norm_G( P)/ P'\right)\left( Q/ P'\right)$;
\item when $\hat A\neq G$, each $\theta\in\Irr_{2'}(\hat A)$ is $\galh_2$-invariant;
\item $A$ acts nontrivially on $Z$ if $Z\neq 1$;
\item any character in $\Irr_{2'}(G)$ that is nontrivial on $Z$ is moved by $\galh_2$ if and only if it is moved by the action of $A$; 
\end{itemize}
and we are able to see from this and the explicit character tables that there exists a bijection $\Omega\colon\Irr_{2'}(G)\rightarrow \Irr_{2'}(\norm_G(P)/P')$ satisfying Condition \ref{cond:condition}. Here we identify $\Irr_{2'}(\norm_G(P)/P')$ with $\Irr_{2'}(\norm_G(P))$. 

We next claim that
$(\hat A, G, \chi)_{\galh_2} \geqslant_c (\norm_{\hat A}(P), \norm_G(P), \Omega(\chi))_{\galh_2}$
for each $\chi\in\Irr_{2'}(G)$, so that by again applying \cite[Theorem 2.9]{NavarroSpathVallejo}, we will have that \eqref{eq:extpart} (and hence the inductive McKay--Navarro condition for the prime $2$) holds. For this, it suffices to consider only the cases that $A/S$ has order $2$. Further, we need only consider those $\chi\in\Irr_{2'}(G)$ that are invariant under $\hat A$, and hence extend to $\hat A$. From our observation that each character in $\Irr_{2'}(\hat A)$ is $\galh_2$-invariant, we see that in particular, any extension of $\chi$ to $\hat A$ is $\galh_2$-invariant. We are left to argue that there is an extension of $\psi:=\Omega(\chi)$ to $\norm_{\hat A}(P)$ that is stable under any $\sigma\in\galh_2$ stabilizing $\psi$. (Note that an extension of $\psi$  exists by the equivariance of $\Omega$ since $|\hat A/G|=2$.) Considering $\psi$ as a character of $\norm_G(P)/P'$, let $\lambda$ be the determinantal character $\lambda:=\det(\psi)$. Then $\lambda$ is a linear character whose restriction $\lambda_0$ to $P/P'$ must also be linear. But since $\psi$ is $Q$-invariant, we see so is $\lambda$, and hence $\lambda_0$. 
In particular, $\lambda_0$ extends to some $\wt\lambda_0$ in $\Irr(Q/P')$, and there is a unique common extension $\mu$ of $\lambda$ and $\wt\lambda_0$ to $\norm_{\hat A}(P)/P'$, by applying \cite[Lemma 4.1]{Spath10}. Note that as a linear character, $\wt\lambda_0$ can be regarded as a character of $Q/Q'$, which must be rational since $Q/Q'$ is elementary abelian. Then $\mu$ is invariant under any $\sigma\in\galh_2$ stabilizing $\lambda$ (and hence under any $\sigma\in\galh_2$ stabilizing $\psi$). By \cite[Lemma 6.24]{isaacs} there is a unique extension of $\psi$ to $\norm_{\hat A}(P)/P'$ with determinantal character $\mu$. This extension must then also be invariant under any $\sigma\in\galh_2$ stabilizing $\psi$, completing the claim.

Finally, we discuss the case of $\operatorname{PSU}_4(3)$, whose Schur multiplier and outer automorphism group are more complicated.
For $S=\operatorname{PSU}_4(3)$, we consider the odd covering group $G:=3^2.S$, and note $\out(S)=\mathrm{Dih}_8$ is the dihedral group of size 8.  In this case, $S$ has a self-normalizing Sylow $2$-subgroup, so for $Q\in\mathrm{Syl}_2(G)$,   we have $N:=\norm_G(Q)=Q\times \zent(G)$. We also remark that the members of $\irr_{2'}(S)$ and $\irr_{2'}(Q)$ are $\galh$-invariant. 
Hence we can see the action of $\galh$ directly from the character table for $G$, and from the action on $C_3\times C_3$ for $N$. We can partition $\irr_{2'}(G)$ and $\irr_{2'}(N)$ into 5 sets each, based on the kernels of the characters when restricted to ${\zent(G)}$, which are either $\zent(G)$ or a copy of $C_3$, and we can see that such a partitioning can be chosen equivariant with respect to the automorphisms that preserve that kernel, using that $S$ satisfies the inductive McKay conditions \cite{malle08b}, 
and we can further see directly from the action of $\galh$ that this is $\galh$-equivariant.  
 
Note that $\aut(S)$ acts faithfully on $\zent(G)$ (see \cite[Table 6.3.1]{GLS3}), and this action is described, for example, in \cite{lindsey}). 
Namely, if we let $\delta$ denote a diagonal automorphism inducing the action of $\operatorname{PGU}_4(3)/\PSU_4(3)$ on $\PSU_4(3)$ and $\varphi$ a field automorphism of order $2$, so that $\mathrm{Dih}_8=\langle\delta, \varphi\rangle$, we have $\delta$ acts on one copy of $C_3$ by inversion and trivially on the other, and $\varphi$ interchanges the copies of $C_3$.  
Further, the members of $\irr_{2'}(S)$ are stable under $\varphi$ and there are two pairs (of degrees 35 and 315) that are interchanged by $\delta$.  Observing the corresponding values for $\irr_{2'}(G)$, we are able to see that it is possible to define our map in such a way that it is $\aut(S)\times\galh$-equivariant.
From here, since $|\out(S)|$ is a $2$-group and $\irr_{2'}(N)$ is comprised of linear characters, we may argue similarly to the case of $\type{B}_2(2)'$ in \cite[Proposition 5.13]{birtethesis} to obtain the desired extensions, completing the proof.  
\end{proof}

\section{Notation and Preliminaries}\label{sec:prelim}

From now on, we set our sights on the groups of Lie type defined in odd characteristic. We begin by setting some notation that we will use throughout. 

\subsection{Characters}\label{characters}
Given finite groups $X\leq Y$, we write $\irr(X)$ for the set of irreducible ordinary characters of $X$, $\irr(X|\chi)$ for the set of irreducible constituents of the restriction $\Res^Y_X(\chi)$ for  $\chi\in\irr(Y)$, and $\irr(Y|\psi)$ for the set of irreducible constituents of the induced character $\Ind^Y_X(\psi)$ for $\psi\in\irr(X)$.  
If $X\lhd Y$ and $\mathcal{X}$ is a subset of $\irr(X)$, an \emph{extension map} for $\mathcal{X}$ with respect to $X\lhd Y$ is a map $\Lambda\colon\mathcal{X}\rightarrow \bigcup_{\psi\in\mathcal{X}}\Irr(Y_\psi|\psi)$ such that for each $\psi\in\mathcal{X}$, the character $\Lambda(\psi)$ is an extension of $\psi$ to $Y_\psi$.

\subsection{Groups of Lie Type}
Throughout, we let $\bG$ be a connected, reductive algebraic group defined in characteristic $p$, and let $q$ be a power of $p$. Let $F\colon \bG\rightarrow\bG$ be a Frobenius endomorphism giving $\bG$ an $\F_q$-rational structure, and write $G:=\bG^F$ for the group of fixed points, giving a finite group of Lie type defined over $\mathbb{F}_q$. (Note that the case of Suzuki and Ree groups has been completed in \cite{birtethesis}, so we do not need to consider the case of more general Steinberg endomorphisms here.)

We let $\bT\leq \bg{B}$ be an $F$-stable maximal torus inside an $F$-stable Borel subgroup, and further write $T:=\bT^F$.  We will call such a $T$ a \emph{maximally split torus of $G$}. We will write $\bg{W}:=\norm_{\bG}(\bT)/\bT$ for the Weyl group of $\bG$, let $W:=\bg{W}^F$, and let $\Phi$ denote the root system for $\bG$ with respect to $\bT$. We use the  notation $n_\alpha(t')$, $h_\alpha(t')$, and $x_\alpha(t)$ to denote the Chevalley generators as in the Chevalley relations \cite[Theorem 1.12.1]{GLS3} for $t\in\overline{\F}_p$, $t'\in\overline{\F}_p^\times$, and $\alpha\in\Phi$.

Letting $N:=\norm_G(\bT)$, we have $N=TV$ and $\norm_{\bG}(\bT)=\langle \bT, V\rangle$ for the so-called \emph{extended Weyl group}  $V:=\langle n_\alpha(1)\mid\alpha\in\Phi\rangle\leq \norm_\bG(\bT)$ defined as in \cite[Section 3.A]{MS16}. Then $H:=V\cap \bT=\langle h_\alpha(-1)\mid\alpha\in\Phi\rangle$ is an elementary abelian $2$-group.  

\subsection{Harish-Chandra Theory}
More generally, we can consider an $F$-stable  Levi subgroup $\bL\leq \bg{P}$ inside an $F$-stable parabolic subgroup of $\bG$, and write $L:=\bL^F$. Here $\bg{P}^F=U\rtimes L$ is a semidirect product with an appropriate subgroup $U$. If $\lambda\in\irr(L)$, we write $\HC_L^G(\la)$ for the Harish-Chandra induced character - that is, the character obtained by inflating $\la$ to $\bg{P}^F$ and then inducing it to $G$. If $\la$ is moreover cuspidal, we call the pair $(\bL, \la)$ a \emph{cuspidal pair} for $G$. In the special case that $L=T$ is a maximally split torus, the irreducible constituents of $\HC_T^G(\la)$ for $\lambda\in\Irr(T)$ are known as the \emph{principal series}.

Given a cuspidal pair $(\bL, \la)$, we have a bijection between the set $\mathcal{E}(G,\bL,\la)$ of irreducible constituents of $\HC_L^G(\la)$ and the irreducible characters of the \emph{relative Weyl group} $W(\la):=\norm_{\bG^F}(\bL)_\la/L$, and we fix such a bijection throughout. (See \cite[Section 10.6]{carter2}).  
With this established, we will denote by $\HC_L^G(\la)_\eta$ the constituent of $\HC_L^G(\la)$ corresponding to the character $\eta\in\irr(W(\la))$.  
Now, the group $W(\la)$, which is not necessarily a reflection group, can be decomposed as a semidirect product $W(\la)=R(\la)\rtimes C(\la)$, where $R(\la)$ is a Weyl group with some root system $\Phi_\la$ and  $C(\la)$ is the stabilizer of a simple system in $\Phi_\la$. (See \cite[Section 10.6]{carter2}.)

\subsection{Duality}
Associated to $(\bG, F)$ is a dual pair $(\bG^\ast, F^\ast)$ with maximally split torus $\bT^\ast$ dual to $\bT$, and we write $G^\ast:=(\bG^\ast)^{F^\ast}$. To ease notation, we write $F$ also for the Frobenius $F^\ast$. Given a semisimple element $s\in G^\ast$, we may define $W^\circ(s)$ to be the Weyl group of $\cent_{\bG^\ast}^\circ(s):=\cent_{\bG^\ast}(s)^\circ$ with respect to $\bT^\ast$ and $W(s)$ to be the Weyl group of $\cent_{\bG^\ast}(s)$ with respect to $\bT^\ast$. Then we have  $W(s)=W^\circ(s)\rtimes A(s)$, where $A(s)\cong \cent_{\bG^\ast}(s)/\cent^\circ_{\bG^\ast}(s)$. (See e.g. \cite[Proposition 1.3]{bonnafequasiisol}.) If $s\in T^\ast:={\bT^\ast}^F$, there is an associated $\la\in\irr(T)$, and the decompositions of $W(\la)$ and $W(s)$ are closely related (this will be made more precise when needed - see e.g. \cite[Lemma 3.4]{SFT22} and \cite[Lemma 4.5]{SFgaloisHC}).

By results of Lusztig \cite{lus88}, the set $\irr(G)$ can be partitioned into subsets $\mathcal{E}(G,s)$ called rational series, indexed by $G^\ast$-conjugacy class representatives of semisimple elements $s\in G^\ast$. For $\chi\in\mathcal{E}(G,s)$, we have $\chi(1)$ is divisible by $[G^\ast:\cent_{G^\ast}(s)]_{p'}$, implying that if $\chi\in\irr_{2'}(G)$, it must be that $s\in G^\ast$ centralizes a Sylow $2$-subgroup of $G^\ast$. Such an element $s$ is often called \emph{$2$-central}.  

\subsection{Automorphisms}
Now, by Section \ref{sec:sporalt}, the remaining simple groups to be considered are simple groups of Lie type with nonexceptional Schur multipliers. Here the universal covering group of such a simple group is of the form $G=\bG^F$ with $\bG$ simple of simply connected type. When $\bG$ is simple of simply connected type and $G$ is perfect (as in the cases we are left to consider), we may obtain the group $\aut(G)$ in a relatively natural way.  Namely, we may take a regular embedding $\iota\colon \bG\rightarrow\wt\bG$ as in \cite[Section 1.7]{GM20} so that, among other properties, $[\bG,\bG]=[\wt{\bG},\wt{\bG}]$ and $\zent(\wt\bG)$ is connected. The Frobenius $F$ may be naturally extended to $\wt\bG$, and we let $F$ still denote this morphism and write $\wt{G}=\wt{\bG}^F$.  Then the automorphisms $\aut(G)$ are induced by $\wt{G}\rtimes {D}$, where  ${D}$ is an appropriate group of graph-field automorphisms of $G$, see \cite[Theorem 2.5.1]{GLS3}. We remark that $D$ can be chosen to act on the set of $F$-stable subgroups of $\wt\bG$.

Further, $\iota$ induces a dual map $\iota^\ast\colon\wt\bG^\ast\rightarrow\wt\bG$ such that the characters in $\mathcal{E}(G, s)$ are the restrictions of the characters in $\mathcal{E}(\wt{G}, \wt{s})$, where $\wt{s}\in\wt{G}^\ast$ satisfies $\iota^\ast(\wt{s})=s$.  The group of linear characters $\irr(\wt{G}/G)$ is isomorphic to $\zent(\wt{G}^\ast)$, and if $\hat z\in\irr(\wt{G}/G)$ corresponds to $z\in\zent(\wt{G}^\ast)$, then $\mathcal{E}(\wt{G}, \wt{s})\otimes \hat{z}=\mathcal{E}(\wt{G}, \wt{s}z)$ (see e.g. \cite[Prop. 11.4.12, Rem. 11.4.14]{DM20}).

\section{Odd-degree characters and rationality}\label{sec:rationality}

In this section, we discuss the action of $\galh:=\galh_2$ on the sets $\irr_{2'}(G)$ for the quasisimple groups $G=\bG^F$ of Lie type where $\bG$ is simple of simply connected type $\type{A}$, $\type{D}$, or $\type{E}_6$.

\subsection{Recalling Previous Results}\label{sec:oldresults}

We begin with some previous results on the sets $\irr_{2'}(G)$ for the groups $G=\bG^F$ under consideration.

A key result from \cite{MS16} (completed later for groups of type $\type{A}$ in \cite{malle19}) is that if $\bG$ is of simply connected type, with some exceptions, any $\chi\in\irr_{2'}(G)$ lies in the principal series. Here the groups $\Sp_{2n}(q)$ are the main exception.

\begin{proposition}[Malle--Sp{\"a}th, Malle \cite{MS16, malle19}]\label{prop:princseries}
Let $G=\bG^F$ with $\bG$ of simply connected type over $\bar\F_q$ with $q$ odd and $F$ a Steinberg endomorphism.  Suppose $G\neq \Sp_{2n}(q)$ for $n\geq 1$, and let $\chi\in\irr_{2'}(G)$.  Then $\chi$ lies in the principal series.  In particular, $\chi=\HC_T^G(\la)_\eta$ with $T$ a maximally split torus of $G$, $\la\in\irr(T)$,  $\eta\in\irr_{2'}(W(\la))$, and $2\nmid[{W}:W(\la)]$.
\end{proposition}

\begin{proof}
That $\chi$ lies in the principal series is \cite[Theorem 7.7]{MS16} and \cite[Theorem 3.3]{malle19}.  The second statement is \cite[Lemma 7.9]{MS16}, which now follows for $\bG$ of type $\type{A}$ as well by the same proof and \cite[Theorem 3.3]{malle19}. 
\end{proof}

We remark that \cite[Theorem 7.7]{MS16} details explicitly the series for $G=\Sp_{2n}(q)$ that contain odd-degree characters.  However, we do not need to consider this here, thanks to the authors' previous work \cite{SFT22, SF22, RSF22}. In \cite[Theorem 3.8]{SFgaloisHC}, the second author details how Galois automorphisms act on the Howlett--Lehrer labelings $\HC_L^G(\la)_\eta$.  Thanks to Proposition \ref{prop:princseries}, we will only need the result here for the principal series.  With this in mind, let $T:=\bT^F$ denote a maximally split torus of $G=\bG^F$, let $\lambda\in\irr(T)$, and let $\eta\in\irr(W(\la))$.

Let $N:=\norm_{G}(\bT)$ and let $\Lambda\colon \irr(T)\rightarrow\bigcup_{\psi\in\irr(T)}\irr(N_\psi)$ be an extension map for $\irr(T)$
with respect to $T\lhd N$, which exists by \cite{geck93} and \cite[Theorem 8.6]{lusztig84}. That is, for each $\psi\in\irr(T)$, we have $\Lambda(\psi)$ is an extension of $\psi$ to $N_\psi$, see Section \ref{characters}.  For $\sigma\in\gal$, let $\delta_{\la,\sigma}$ be the linear character of $W(\la)$ such that $\Lambda(\la)^\sigma=\delta_{\la,\sigma}\Lambda({\la^\sigma})$, guaranteed by Gallagher's theorem \cite[Corollary 6.17]{isaacs}, and let $\delta'_{\la,\sigma}\in\irr(W(\la))$ be the character such that $\delta'_{\la,\sigma}(w)=\delta_{\la,\sigma}(w)$ for $w\in C(\la)$ and $\delta'_{\la,\sigma}(w)=1$ for $w\in R(\la)$.  

Recall that $W(\la)\leq \bg{W}$ since $\la\in\Irr(T)$. Let $\gamma_{\la, \sigma}$ be the character of $W(\la)$ such that $\gamma_{\la, \sigma}$ is trivial if $\sigma$ fixes $\sqrt{q}$ and otherwise satisfies $\gamma_{\la,\sigma}(w)=(-1)^{l(w_c)}$ for $w=w_rw_c$ with $w_r\in R(\la)$ and $w_c\in C(\la)$, where $l(w_c)$ is the length $l(w_c)$ in $\bg{W}$ of $w_c$.  (Note that $\gamma_{\la,\sigma}$ is defined more generally in \cite{SFgaloisHC}, but can be viewed as stated here in the case of finite groups of Lie type, see \cite[Lemma 3.5]{SFT22}.)
 
 For $\eta\in\irr(W(\la))$, we denote by $\eta^{(\sigma)}$ the character $\mathfrak{f}(\wt{\eta}^\sigma)$, where $\mathfrak{f}$ is an isomorphism between $\irr(\mathrm{End}_{\overline{\Q}G}(\HC_T^G(\la)))$ and $\irr(W(\la))$ induced by the standard specializations  and $\wt{\eta}\in\irr(\mathrm{End}_{\overline{\Q}G}(\HC_T^G(\la)))$ is such that $\mathfrak{f}(\wt{\eta})=\eta$.  (See \cite[Section 3.5]{SFgaloisHC}.) Note that this is not necessarily the same as $\eta^\sigma$, although we see in \cite[Section 5]{SF22}, combined with Lemma \ref{lem:etachar} below, that this is the case in the situations in which we are interested.  The following is \cite[Theorem 3.8]{SFgaloisHC} in our situation.

\begin{theorem}[Schaeffer Fry\cite{SFgaloisHC}]\label{thm:GaloisAct}
Keep the notation above and let $\sigma\in\gal$ and  $\chi=\HC_T^G(\la)_\eta$. 
  Then $\chi^\sigma=
\HC_T^G(\lambda^\sigma)_{\eta'}$, where $\eta'\in\irr(W(\la))$ is defined by $\eta'(w)=\gamma_{\la, \sigma}(w)\delta'_{\la,\sigma}(w^{-1})\eta^{(\sigma)}(w)$ for each $w\in W(\lambda)$. (Recall here that $W(\la)=W(\la^\sigma)$.)
\end{theorem}

From this, we see that the main obstruction to understanding the action of $\gal$ on $\irr_{2'}(G)$ is in understanding the characters $\eta^{(\sigma)}, \gamma_{\la,\sigma}$, and $\delta_{\la,\sigma}$ (or $\delta_{\la,\sigma}'$).  Significant progress was made on this in \cite{SF22, SFT22}.  In  the next subsections, we extend these results to the cases that $\bG$ is of type $\type{A}$ or $\type{D}$ omitted there.

The following will also be useful for the case that $\bG$ is of type $\type{A}$.  As usual, we use $\SL_n^\epsilon(q)$ with $\epsilon\in\{\pm\}$ to denote $\SL_n(q)$ in the case $\epsilon=+$ and $\operatorname{SU}_n(q)$ in the case $\epsilon=-$, and similar for the groups $\GL_n^\epsilon(q)$.

\begin{lemma}
\label{lem:typeArest}
Let $q$ be odd and let $G=\SL_n^\epsilon(q)$ and $\wt{G}=\GL_n^\epsilon(q)$ with $n\geq 3$.  Let $\chi\in\irr_{2'}(G)$ and $\wt{\chi}\in\irr(\wt{G}|\chi)$. 
Then either $\Res^{\wt{G}}_G(\wt{\chi})=\chi$ or we have $\Res^{\wt{G}}_G(\wt{\chi})=\chi+\chi'$ is the sum of two irreducible characters and $n=2^r$ is some power of $2$. In the first case, ${\chi}\in\mathcal{E}({G},{s})$ with $\cent_{\bG^\ast}({s})$ connected, and in the second case, one can choose $\wt{\chi}\in\mathcal{E}(\wt{G},\wt{s})$ with $\wt{s}=\mathrm{diag}(1,\ldots, 1, -1,\ldots,-1)$ with $2^{r-1}$ copies of each eigenvalue.
\end{lemma}
\begin{proof}
The first statement is \cite[Lemma 10.2]{SFT18a}, which is shown as a consequence of \cite[Lemmas 4.5 and 4.6]{NT16}, and the second statement follows from \cite[Lemmas 4.5 and 4.6]{NT16}.
\end{proof}

\subsection{The Characters $\eta^{(\sigma)}$ and  $\gamma_{\la,\sigma}$}\label{sec:etagammadelta}

In this section, we now let $\bG$ be simple and simply connected (of arbitrary type, unless otherwise specified) and continue to write $G:=\bG^F$, where now  $F\colon \bG\rightarrow\bG$ is a Frobenius endomorphism defining $G$ over $\F_q$ with $q$ odd.

Thanks to the statements in Section \ref{sec:oldresults} and the work of \cite[Section 5]{SF22}, we can say that in our situation, the character $\eta^{(\sigma)}$ is just $\eta$:
\begin{lemma}\label{lem:etachar}
Let $\bG$ be simply connected of type $\type{A}_{n-1}$ with $n\geq 3$ or $\type{D}_n$ with $n\geq 4$ and write $G=\bG^F$.  Let $\chi\in\irr_{2'}(G)$, so that $\chi= \HC_T^G(\la)_\eta$ as in \prettyref{prop:princseries}.  Then $|C(\la)|\leq 2$ or $\bG$ is of type $\type{D}_n$ and $C(\la)$ is elementary abelian of order $4$.  
\end{lemma}
\begin{proof}
Let $s\in G^\ast$ be semisimple such that $\chi\in\mathcal{E}(G, s)$.  Let $\iota\colon\bG\hookrightarrow\wt{\bG}$ be a regular embedding as in \cite[Section 1.7]{GM20}, let $\wt{G}:=\wt{\bG}^F$, and let $\wt{s}\in\wt{G}^\ast$ be such that $\iota^\ast(\wt{s})=s$.  Further, let $\wt{\la}\in\irr(\wt{T}|\la)$, where $\wt{T}$ is a maximally split torus of $\wt{G}$ containing $T$.  Then by \cite[Lemma 3.4]{SFT22} and \cite[Lemma 4.5]{SFgaloisHC}, we have $W(\la)\cong W(s)^{F}$ and $R(\la)=W(\wt{\la})=R(\wt{\la})\cong W(\wt{s})^{F}=W^\circ(s)^{F}$, 
where the last equality follows by \cite[(2.2)]{bonnafequasiisol}.  In particular, we have $C(\la)$ is isomorphic to $A(s)^F$, where $A(s)=W(s)/W^\circ(s)\cong \cent_{\bG^\ast}(s)/\cent_{\bG^\ast}^\circ(s)$.  

For $\bG$ of type $\type{D}$, the result about $C(\la)$ now follows from the description of disconnected $\cent_{\bG^\ast}(s)$ for $2$-central semisimple elements $s\in\bG^\ast$ in \cite[Table 1]{MS16} and \cite[Table 2]{malle19}, which follows from \cite[Table 2]{bonnafequasiisol}.
For $\bG$ of type $\type{A}$ the statement about $C(\la)$ follows since $|A(s)^F|\leq 2$ using \prettyref{lem:typeArest}.
\end{proof}

\begin{corollary}\label{cor:etachar}
 Let $G=\bG^F$ where $\bG$ is simple of simply connected type. Let $\chi\in\irr_{2'}(G)$ lie in the principal series, so that it is of the form $\chi=\HC_T^G(\la)_\eta$ as before.  Then $\eta^{(\sigma)}=\eta^\sigma=\eta$ for each $\sigma\in\gal$.
\end{corollary}
\begin{proof}
For the groups not of type $\type{A}$ and $\type{D}$, this was established in \cite[Sections 5 and 8]{SF22} using \cite[Corollary 5.6]{SF22}. The full statement can now be seen using Lemma \ref{lem:etachar} and \cite[Corollary 5.6]{SF22}.  We remark that the result still holds for $\type{A}_1$ in the case of principal series characters.
\end{proof}

We next consider the character $\gamma_{\la,\sigma}$.  

 \begin{proposition}\label{prop:gammachar}
 Let $G=\bG^F$ where $\bG$ is simple of simply connected type, not of type $\type{A}_1, \type{C}_n$. Let $\chi=\HC_T^G(\la)_\eta\in\irr_{2'}(G)$ as before.  Then $\gamma_{\la,\sigma}=1$ for every $\sigma\in\gal$.
 \end{proposition}
 \begin{proof}
 For $\bG$ not of type $\type{A}_{n-1}$, this is \cite[Proposition 3.4]{SF22}.  So assume that $\bG$ is type $\type{A}_{n-1}$ with $n\geq 3$. If $C(\la)=1$, then certainly the statement holds. Otherwise, we have $n=2^r\geq 4$ and $C(\la)$ has size $2$ by Lemma \ref{lem:etachar}.  
 From the description there of $\wt{s}$ and $R(\la)$, we see that $R(\la)$ is a subgroup of $\sym_{2^{r-1}}\times \sym_{2^{r-1}}$, with $C(\la)$ inducing the element $\prod_{i=1}^{2^{r-1}}(i, 2^{r-1}+i)$, which has even length in $\bg{W}$ since $r\geq 2$. Hence the statement follows from  \cite[Lemma 4.10]{SFgaloisHC}.
 \end{proof}

	\subsection{Extension Maps and Galois Automorphisms}\label{sec:extmaps}

In this section, we assume that $\G$ is simple of simply connected type $\type{A}_n$, $\type{D}_n$ ($n \geq 2$) or $\type{E}_6$.
Let $\T$ be a maximally split torus of $(\bG,F)$.
We consider the stabilizers of $2$-central elements in $G^\ast$.

Throughout this section, we will work with a semisimple $s\in T^\ast$ and the associated linear character $\la\in\irr(T)$, where $\la$ corresponds to $s$ under the isomorphism $\Irr(T) \to T^\ast$ induced by duality. Now, for $\sigma\in\gal$ and $s\in T^\ast$, we may define $s^\sigma$ to be $s^k$, where $\sigma$ maps $|s|$'th roots of unity to their $k$th power. (For $\sigma\in\galh$, we may be more specific, see \cite[Lemma 3.4]{SFT18a}.) Further, $T$ is $D$-stable, and given $\alpha\in D$, we may find a dual automorphism $\alpha^\ast $ of $G^\ast$, and following the proofs of \cite[Lemma 3.4]{SFT18a} and \cite[Proposition 7.2]{Taylor}, we see that the isomorphism $\Irr(T)\rightarrow T^\ast$ is compatible with the action of $D\times \gal$.  

We will write $\odd(T)$ for the set of $\la\in\irr(T)$ such that the corresponding $s\in G^\ast$ is $2$-central. The structure of $2$-central elements is given in the proof of \cite[Theorem 3.4]{malle19} (see also Lemma \ref{lem:typeArest}). We further let $\galh:=\galh_2$.

\begin{lemma}\label{stab}
Keep the notation and assumptions above. Then every $\lambda\in\odd(T)$ is $\N_G(\T)$-conjugate to some $\lambda_0 \in \odd(T)$ satisfying
$$(\N_{G {D}}(\T)\times \mathcal{H})_{\lambda_0}=\N_{G}(\T)_{\lambda_0} ({D} \times \mathcal{H})_{\lambda_0}.$$
\end{lemma}

\begin{proof}
	We assume first that $G$ is of type $\type{A}$.
	If $A(s) \neq 1$ then by \cite[Theorem 3.4]{malle19}, $s$ is $\N_{G^\ast}(\T^\ast)$-conjugate to $s_0:=\mathrm{diag}(-1,\dots,-1,1,\dots,1)$. In particular, the character $\lambda_0 \in \Irr(T)$ corresponding to $s_0$ via duality is $D \times \mathcal{H}$-stable which gives the statement.
	If $A(s)=1$ we have that $s$ is $\N_{G^\ast}(\T^\ast)$-conjugate to $s_0=\mathrm{diag}(\mu_1,\dots,\mu_2,\dots,)$ where we assume that each distinct eigenvalue $\mu_i$ appears $n_i$ times with $n_1 \geq n_2 \geq \dots \geq n_k$. Since $s$ (and therefore $s_0$) is $2$-central it follows that all $n_i$ must necessarily be distinct. Any automorphism of $D \times \galh$ stabilizing the $\N_{G^\ast}(\T^\ast)$-orbit of $s_0$ permutes its eigenvalues. Therefore, it must stabilize $s_0$. Hence, the character $\lambda_0 \in \Irr(T)$ corresponding to $s_0$ satisfies the conclusion of the statement.
	
	For type $\type{D}$, we observe by \cite[Lemma 7.5]{MS16} and \cite[Proposition 4.11, Table II]{bonnafequasiisol} that every $1 \neq s$ is quasi-isolated with disconnected centralizer and $s^2=1$. Moreover, by \cite[Table 2]{malle19} and \cite[Table 1]{MS16} the associated character $\lambda$ is $\N_G(\T)$-conjugate to a ${D} \times \mathcal{H}$-stable character.

	For $\G$ of type $\type{E}_6(\varepsilon q)$ we observe any $2$-central elements $s \in T^\ast$ is $\N_{G^\ast}(\T)$-conjugate to an element $s_0$ contained in the center $\Z(\Levi^\ast)$ of an $F$-stable Levi subgroup $\Levi^\ast$ of $\G^\ast$ containing $\T^\ast$ of rational type $\type{D}_5(\varepsilon q) (q-\varepsilon)$, see the proof of \cite[Theorem 3.4]{malle19}. Here, $L^\ast:=(\Levi^\ast)^F=\C_{G^\ast}(t)$ is the centralizer of an involution $t \in T^\ast$ which can be chosen to be $D$-stable by the proof of \cite[Theorem 3.4]{malle19}. In particular, we can assume that $L^\ast$ is $D$-stable. The relative Weyl group $\N_{\G^\ast}(\Levi^\ast)/\Levi^\ast$ is trivial and we have $L^\ast=\C_{\G^\ast}(s_0)$ if $s_0 \neq 1$. Hence, any automorphism of $D \times \galh$ which stabilizes the $\N_{G^\ast}(\T)$-orbit of $s_0\in \Z((\Levi^\ast)^F)$ must therefore centralize $s_0$. This implies again the stabilizer condition for the character $\lambda_0$ corresponding to $s_0$.
\end{proof}

\begin{lemma}\label{lem:firstextmapA}
Keep the notation and assumptions above, and further assume $\G$ is not of type $\type{D}$. Then there exists an $(\N_{G {D}}(\T)\times \gal)$-equivariant extension map $\Lambda$ for $\odd(T)$ with respect to $T\lhd N$. 
\end{lemma}

\begin{proof}
	Let $s$ be the semisimple element associated to $\lambda \in \odd(T)$. According to \cite[Theorem 3.4]{malle19}, $\C^\circ_{\G^\ast}(s)$ is an $F$-stable Levi subgroup of $\G^\ast$ and we let $\T \subset \Levi$ be the $F$-stable Levi subgroup of $\G$ corresponding to it under duality. The character $\hat{s} \in \Irr(\Levi^F)$, where $\hat{s}$ is the character corresponding to $s$ under duality, is an extension of the character $\lambda$ and $W(\lambda)$ corresponds to $W(s)$ under the isomorphism $W \cong W^\ast$ induced by duality. (See, e.g. \cite[Lemma 4.5]{SFgaloisHC}.) In particular, if $\C_{\G^\ast}(s)$ is connected, then we can define $\Lambda(\lambda):=\Res^L_{\N_G(T)_\la}(\hat{s})$ and this character has the desired properties.

	 We can therefore assume that $\C_{\G^\ast}(s)$ is disconnected. In this case, $\G$ is of type $\type{A}_{n-1}$ with $n$ a power of $2$ and $\lambda \in \Irr(\T^F)$ is $\N_G(\T)$-conjugate to the linear character $\lambda_0$ associated to $s_0=\mathrm{diag}(-1,\dots,-1,1,\dots,1)$. Then $\N_G(\Levi)_{\lambda_0}=L \langle v \rangle$, where $v\in V$ is the canonical representative in $V$ of the longest element of the Weyl group as in \cite[Section 3.A]{MS16}. In particular, $v$ is $D$-stable and $v^2=-1$. Since $\lambda_0$ is trivial on $\Z(G)$, the character $\hat{s}_0 \in \Irr(L)$ extends to a character $\hat{s}'_0 \in \Irr(L \langle v \rangle \mid \hat{s})$ with $v$ in its kernel. As $v$ is $D$-stable and $\hat{s}'_0$ is a linear character of order $2$, it follows that $\Lambda(\la_0):=\Res^{L \langle v \rangle}_{\N_G(\T)_{\la_0}}$ is $D \times \mathcal{H}$-stable. Since $\la=\la_0^m$ for some $m \in \N_G(\T)$ we can then define $\Lambda(\la):=\Lambda(\la_0)^{m}$.
\end{proof}

We remark that the omitted case of type $\type{D}$ is addressed below in  \prettyref{prop:deltatypeD}.

\bigskip

From now on through the rest of the paper, we let $\Lambda$ be the $N{D}$-equivariant extension map for $\Irr_{\mathrm{2cen}}(T)$ with respect to $T\lhd N$ discussed in  \cite[Lemma 4.1]{SF22} (which exists by \cite[Prop. 5.9]{CS16}, \cite[Cor. 3.13]{MS16}, \cite[Thm. 3.6]{CS13}, and \cite{spath09}) if $\bG$ is not of type $\type{A}_n$, and as in Lemma \ref{lem:firstextmapA} if $\bG$ is type $\type{A}_n$. We next consider the character $\delta_{\la, \sigma}$.  Recall that this is the linear character of $W(\la)$ such that $\Lambda(\la)^\sigma=\Lambda(\la^\sigma)\delta_{\la,\sigma}$.  In \cite[Lemma 4.3]{SF22}, this character is studied and $\Lambda$ is shown to be ``sufficiently equivariant" (in the sense that it satisfies the conclusion of the next lemma), with the exception of type $\type{A}$.  Thanks to Lemma \ref{lem:firstextmapA}, we now can complete this omitted case. As in loc. cit., we will sometimes write $\Lambda_\lambda:=\Lambda(\lambda)$
and we define $R_\la:=\langle T, n_\alpha(-1)\mid \alpha\in\Phi_\la\rangle\leq N_\la$ so that $R(\la)\cong R_\la/T$. 

\begin{lemma}\label{lem:typeAextRla}
Assume that $\bG$ is simple of simply connected type, but not of type $\type{A}_1$.
Let $\la\in\odd(T)$ and $\sigma\in\gal$.   Then $\Lambda_\la^\sigma\downarrow_{R_\la}=\Lambda_{\la^\sigma}\downarrow_{R_\la}$.  In particular, $R(\la)\leqslant \ker\delta_{\la,\sigma}$.  
\end{lemma}
\begin{proof}
For $\bG$ not of type $\type{A}$, this is \cite[Lemma 4.3]{SF22}. So let $\bG$ be type $\type{A}_{n-1}$ with $n\geq 3$. Then in this case, the statement follows from \prettyref{lem:firstextmapA}.
\end{proof}

The next proposition complements \prettyref{lem:firstextmapA}.

\begin{proposition}\label{prop:deltatypeD}
Let $G=\bG^F$ with $\bG$ simple of simply connected  type $\type{D}_n$ for $n\geq 4$, and let $\chi=\HC_T^G(\la)_\eta\in\irr_{2'}(G)$. Then the extension $\Lambda_\la$ of $\la$ to $N_\la:=\norm_G(\bT)_\la$ satisfies $\Lambda_\la^\sigma=\Lambda_\la$, and hence $\delta_{\la,\sigma}=1$, for any $\sigma\in\gal$.
\end{proposition}
\begin{proof}
Let $\chi\in\mathcal{E}(G,s)$.  Then by \cite[Lemmas 7.5 and 7.6]{MS16}, \cite[Table 2]{malle19}, we see that $s^2=1$ (and hence $\la^2=1$). We may then assume $G\neq \tw{3}\type{D}_4(q)$, as in this case the result follows from Lemma \ref{lem:typeAextRla} and the fact that $C(\la)=1$.
In the remaining cases, we see from loc. cit. that $A(s)$ is either trivial; or it induces the graph automorphisms simultaneously on the two components of $\cent_{\bG^\ast}^\circ(s)\cong \type{D}_k\times \type{D}_{n-k}$ for some $1\leq k\leq n/2$; or $n=4m$, $\cent_{\bG^\ast}^\circ(s)\cong \type{D}_{2m}\times \type{D}_{2m}$, and $A(s)\cong C_2^2$ is generated by an element inducing the graph automorphisms simultaneously on the two factors of $\cent^\circ_{\bG^\ast}(s)$ and by an element that permutes the two factors.  

Let $\sigma\in\gal$.  Since $\la^2=1$ and hence $\la^\sigma=\la$, and since $R(\la)\leq \ker\delta_{\la,\sigma}$ by Lemma \ref{lem:typeAextRla}, it suffices to show that $\Lambda_\la(\bar c)$ is fixed by $\sigma$ for an inverse image $\bar{c}$ of any $1\neq c\in C(\la)$ in $N$. Since $\chi$ has odd degree and $\zen(G)$ is a $2$-group, we see $\chi$, and hence $\lambda$ (see e.g. \cite[Proposition 2.7(i)]{SFT22}), is trivial on $\zen(G)$.  Then it suffices to know that such a $\bar c$ satisfies $\bar c^2\in\zen(G)$, as then $\Lambda_\la(\bar c)^2=\Lambda_\la(\bar c^2)=\la(\bar c^2)=1$, so $\Lambda_\la(\bar c)\in\{\pm1\}$ is rational.  Since such a $c$ can be represented as the image in $W$ of the element \[n_{\alpha_1}(1)n_{\alpha_2}(1)\cdots n_{\alpha_{n-2}}(1)n_{\alpha_n}(1)n_{\alpha_{n-1}}(1)n_{\alpha_{n-2}}\cdots n_{\alpha_1}(1)\] (this also corresponds to the analogous element $u_1$ in the notation of \cite[7.2]{SFT22} for the case of $\SO_{2n}^\pm(q)$), the element $\prod_{i=1}^{n/2}n_{e_i-e_{n/2+i}}(1)$, or a product of the two, we see by computation with the Chevalley relations \cite[Theorem 1.12.1]{GLS3} that this is indeed the case. 
\end{proof}

\begin{corollary}\label{cor:Doddcharsrational}
Let $G=\bG^F$ with $\bG$ simple of simply connected  type $\type{D}_n$ with $n\geq 4$ and let $\chi\in\irr_{2'}(G)$. Or, let $\bG$ be simple of simply connected type $\type{A}_{n-1}$ with $n=2^{r}\geq 4$ and let $\chi\in\irr_{2'}(G)$ be such that $\chi$ does not extend to $\wt{G}$. Then $\chi$ is rational-valued. 
\end{corollary}
\begin{proof}
Recall that $\chi$ is of the form $\HC_T^G(\la)_\eta$ for $\la\in\odd(T)$ and $\eta\in\Irr_{2'}(W(\la))$ by Proposition \ref{prop:princseries}. Further, recall that $\lambda^2=1$ in the case $\bG$ is type $\type{D}_n$, and similarly the same is true by \prettyref{lem:typeArest} for the characters in type $\type{A}_{n-1}$ with $n=2^r$ that do not extend to $\wt{G}$.  
Then by \prettyref{thm:GaloisAct}, for $\sigma\in\gal$, we have $\chi^\sigma=\HC_T^G(\la^\sigma)_{\eta'}=\HC_T^G(\la)_{\eta'}$. But by combining \prettyref{cor:etachar} with Propositions \ref{prop:gammachar} and \ref{prop:deltatypeD} and  \prettyref{lem:firstextmapA}, we have $\eta'=\eta$, and hence $\chi^\sigma=\chi$ for each $\sigma\in\gal$.
\end{proof}

We remark that \prettyref{cor:Doddcharsrational} already shows that the McKay--Navarro conjecture holds for the groups of type $\type{D}_n$, complementing \cite[Theorem B]{SF22}, since a Sylow $2$-subgroup $P$ of $G$ is self-normalizing in this case by the main results of \cite{Kon05} and satisfies $P/[P,P]$ is elementary abelian by \cite[Theorem A]{NT16} combined with Corollary \ref{cor:Doddcharsrational}.

\subsection{More on Stabilizers and Extensions}

\begin{proposition}
Let $\bG$ be simple of simply connected type $\type{E}_6$, $\type{A}_{n-1}$ with $n\geq 3$, or type $\type{D}_n$ with $n\geq 4$. Let $\chi\in\irr_{2'}(G)$ and let $\galh:=\galh_2$. Then $(\wt{G}\times\galh)_\chi=\wt{G}_\chi\times\galh_\chi$.
\end{proposition}
\begin{proof}
Let $\chi\in\irr_{2'}(G)\cap \mathcal{E}(G,s)$, with $\la\in\irr(T)$ corresponding to $s$, so that $\chi=\HC_T^G(\la)_\eta$ for some $\eta\in\irr_{2'}(W(\la))$ as before. Let $\alpha\in\wt{G}$ and $\sigma\in\galh$ such that $\alpha\sigma\in(\wt{G}\times\galh)_\chi$.
If $\chi$ is as in one of the situations of \prettyref{cor:Doddcharsrational}, then $\galh_{\chi^\alpha}=\galh$ for any $\alpha\in\wt{G}$, and the result follows. 

If $\bG$ is type $\type{E}_6$, then $\cent_{\bG^\ast}(s)$ is a Levi subgroup from before, and hence connected, so $\chi$ extends to $\wt{G}$ and the statement again follows, as $\chi=\chi^{\alpha\sigma}=\chi^\sigma$.  The same holds if $\bG$ is of type $\type{A}_{n-1}$ and $\chi$ extends to $\wt{G}$, completing the proof by Lemma \ref{lem:typeArest}.
\end{proof}

\begin{remark}\label{rem:stabs}
Let $\bG$ be as above.  Recall that for $\bG$ of type $\type{D}_n$ or when $\chi\in\irr_{2'}(G)$ with $G$ of type $\type{A}_n$ and $\chi$ does not extend to $\wt{G}$, we have that $\chi$ is $\galh$-invariant, by  \prettyref{cor:Doddcharsrational}.  In particular,  we have 
\[(\wt{G}D\times\galh)_\chi=(\wt{G}D)_\chi\times\galh_\chi\]
in those cases.

In the remaining cases being considered, $\chi$ extends to $\wt{G}$, so 
\[(\wt{G}D\times\galh)_\chi=\wt{G}(D\galh)_\chi.\]
\end{remark}

\section{Equivariant Bijections for the Prime $2$}\label{sec:equiv}

In previous work, the second-named author \cite{SF22} showed that the bijections used by Malle and Sp{\"a}th in the proof of Isaacs--Malle--Navarro's inductive McKay conditions for $\ell=2$ are also equivariant with respect to $\galh_2$, thus proving Condition \ref{cond:condition}, with a few exceptions.  The aim of this section is to complete the case of the main exceptions in loc. cit.:  groups with root systems of type $\type{A}$ and $\type{D}$.

\begin{theorem}\label{thm:equivbij}
Let $q$ be a power of an odd prime and let $G$ be a group of Lie type of simply connected type defined over $\F_q$ of one of the following types: $\type{A}_{n-1}(q), \tw{2}\type{A}_{n-1}(q)$ with $n\geq 3$ or $\type{D}_n(q), \tw{2}\type{D}_n(q)$ with $n\geq 4$.
Then \prettyref{cond:condition} from Section \ref{sec:iMN} holds for $G$ and the prime $\ell=2$.
\end{theorem}
\begin{proof} 

Throughout, let $\galh:=\galh_2$. First, suppose that $q\equiv 1\mod 4$.
By \cite[Theorem 6.2]{SF22}, we may assume that $\bG$ is type $\type{A}_{n-1}$. With the work in the previous sections, the proof follows exactly as in \cite[Prop. 6.1 and Thm. 6.2]{SF22}, but we provide it for completeness. 

For $\la\in\irr(T)$, recall that the Harish-Chandra series $\mathcal{E}(G, T,\la)$ is the set of irreducible constituents of $\HC_T^G(\la)$ and that $\Lambda$ is an $N{D}$-equivariant extension map with respect to $T\lhd N$, which is further $\galh$-equivariant  
on $\odd(T)$ by \prettyref{lem:firstextmapA}. Let $\wt{N}:=\N_{\wt{G}}(\bT)$. 
From \cite[Theorem 5.2]{MS16} (using the corresponding results for type $\type{A}$ in \cite{CS16}), the map 
\[\Omega\colon\bigcup_{\la\in\irr(T)} \mathcal{E}(G, T,\la)\rightarrow \irr(N) \] given by 
\begin{equation}\label{eq:bij}
\HC_T^G(\la)_\eta \mapsto \Ind_{N_{\la}}^N(\Lambda(\la)\eta)
\end{equation} defines an $\wt{N}{D}$-equivariant bijection, which restricts to a bijection (which we will also call $\Omega$) between $\irr_{2'}(G)$ and $\irr_{2'}(N)$, using \prettyref{prop:princseries}. Further, the group $M:=N$ and map $\Omega$ satisfy the stated properties (other than $\galh$-equivariance) by \cite[Theorem 7.8]{malle07} and \cite[Proposition 2.5]{CS13}.  

Now, we have 
\[\Omega(\chi^\sigma)=\Omega(\HC_T^G(\la^\sigma)_{\gamma_{\la,\sigma}(\delta'_{\la,\sigma})^{-1}\eta^{(\sigma)}})=\Ind_{N_\la}^N\left(\Lambda(\la^\sigma)\gamma_{\la,\sigma}(\delta'_{\la,\sigma})^{-1}\eta^{(\sigma)}\right),\]
by \prettyref{thm:GaloisAct} and the definition \eqref{eq:bij} of $\Omega$.
  But $\eta^{(\sigma)}=\eta^\sigma=\eta$,  $\gamma_{\la,\sigma}=1$, and $\delta_{\la,\sigma}=\delta_{\la,\sigma}'=1$, by \prettyref{cor:etachar}, \prettyref{prop:gammachar}, and 
  \prettyref{lem:firstextmapA}, respectively. Hence, we have 
  \[\Omega(\chi^\sigma)
  =\Ind_{N_\la}^N\left(\Lambda(\la^\sigma)\delta_{\la,\sigma}\eta^{\sigma}\right)
=\Ind_{N_{\la}}^N(\Lambda(\la)^\sigma\eta^\sigma)=
 \Omega(\chi)^\sigma,\]
  completing the proof when $q\equiv 1\mod 4$.

We are now left to consider the case that $q\equiv 3\pmod 4$. Recall that $\bG$ is type $\type{A}_{n-1}$  with $n\geq 3$ or $\type{D}_n$ with $n\geq 4$. We use the same notation and considerations made in \cite[Section 8]{SF22}, which we now summarize.  

We may write $\bg{T}=\cen_{\bg{G}}(\bg{S})$ for some Sylow $2$-torus $\bg{S}$ (as defined in \cite[Section 3.1]{malle07}) of $(\bg{G}, vF)$, where $v\in V$ is a representative for the longest element of $\bg{W}$ as in \cite[3.A]{MS16}. Writing $T_1:=\bg{T}^{vF}$ and $N_1:=\norm_{\bg{G}}(\bg{S})^{vF}=\bg{N}^{vF}$, \cite[Theorem 7.8]{malle07} yields that $N_1$ contains $\norm_G(Q)$ for a Sylow $2$-subgroup $Q$ of $G$ (which we now identify with the isomorphic group $\bG^{vF}$) and there is a bijection 
\[ \Omega_1\colon \irr_{2'}(G)\rightarrow\irr_{2'}(N_1),\]  where corresponding characters lie over the same character of $\zen(G)$.    Then $N_1$ is also $\Gamma=\aut(G)_Q$-stable by \cite[Proposition 2.5]{CS13}), and $\Omega_1$ is moreover $\Gamma$-equivariant by  \cite[Theorem 6.3]{MS16} and \cite[Theorem 6.1]{CS16}, combined with \cite[Theorem 2.12]{spath12}.
 Hence, as in \cite[Section 8]{SF22} it suffices to show that this bijection can be chosen to further be $\galh$-equivariant.

For $\la_1\in\irr(T_1)$, we  write $W_1(\la_1)$ for the group $(N_1)_{\la_1}/T_1$ and $R_1(\la_1)$ for the reflection group generated by the simple reflections $s_\alpha$ for $\alpha\in\Phi$ such that $s_\alpha\in W_1(\la_1)$ and $T_1\cap \langle X_{\alpha}, X_{-\alpha}\rangle$ is in the kernel of $\la_1$.  Here $X_\alpha$ denotes the root subgroup of $\bG$ associated to $\alpha$. Then by
the proof of \cite[Theorem 7.8]{malle07},  both $\irr_{2'}(\bG^{vF})$ and $\irr_{2'}(N_1)$ are parameterized by pairs $(\lambda_1, \eta_1)$ or $(s, \eta_1)$, where $s\in T_1^\ast$ is a semisimple element (up to $N_1^\ast:=(\bg{N}^\ast)^{vF}$-conjugation) centralizing a Sylow $2$-subgroup of ${\bG^\ast}^{(vF)}$; $\la_1\in\mathcal{E}(T_1, s)$ with $2\nmid [N_1:(N_1)_{\la_1}]$; and $\eta_1\in\irr_{2'}(W_1(\la_1))$.  Letting $W_1(s)$ and $W_1^\circ(s)$ denote the fixed points under $vF$ of the Weyl groups of  $\cen_{\bG^\ast}(s)$ and $\cen_{\bG^\ast}(s)^\circ$, respectively, we have 

\begin{equation}\label{eq:W_1R_1}
W_1(\la_1)\cong W_1(s); \quad W_1^\circ(s)\cong  R_1(\la_1); \quad\text{ and }\quad W_1(s)/W_1^\circ(s)\cong W_1(\la_1)/R_1(\la_1).
\end{equation}

The member of $\irr_{2'}(N_1)$ corresponding to $(\la_1, \eta_1)$ is of the form $\Ind_{(N_1)_{\la_1}}^{N_1}(\Lambda_1(\la_1)\eta_1)$, where $\Lambda_1$ is an extension map with respect to $T_1\lhd N_1$. (Note that $\Lambda_1$  exists by \cite[Proposition 5.9]{CS16}.)

First suppose  that $\bG$ is type $\type{D}_n$ with $n\geq 4$.  Then by \prettyref{cor:Doddcharsrational} and the discussion above, it suffices to show that every member of $\irr_{2'}(N_1)$ is rational. But in this case, recall that $s^2=1$, so $\la_1$ is fixed by $\gal$.  Further, the same arguments as in  \prettyref{prop:deltatypeD} yield that $\Lambda_1(\la_1)$ will be rational-valued, as $\Lambda_1(\la_1)$ is also rational-valued on the elements $n_\alpha(-1)$ for $s_\alpha\in R_1(\la_1)$ since $\la_1$ is trivial on $h_\alpha(-1)$ for such elements.  Further, using \eqref{eq:W_1R_1} and the proof of  \prettyref{lem:etachar}, we have $W(\la_1)$ is the semidirect product of the reflection group $R(\la_1)$ with an elementary abelian $2$-group.  As a Weyl group, every character of $R(\la_1)$ must be rational-valued (see e.g. \cite[Theorems 5.3.8, 5.4.5, 5.5.6, and Corollary
5.6.4]{GP00}), and therefore we may apply \cite[Lemma 5.5]{SF22} to see that any $\eta_1\in\irr_{2'}(W_1(\la_1))$ is necessarily fixed by $\gal$. This yields that $\psi:=\Ind_{(N_1)_{\la_1}}^{N_1}(\Lambda_1(\la_1)\eta_1)$ is fixed by $\gal$ for each $\psi\in\irr_{2'}(N_1)$, and we are done with this case.

Finally, let $\bG$ be of type $\type{A}_{n-1}$ with $n\geq 3$. Let $\chi=\HC_T^G(\la)_\eta\in\irr_{2'}(G)\cap\mathcal{E}(G,s)$ and let $\sigma\in\galh$.  As before, we have $\eta^{(\sigma)}=\eta$ and $\gamma_{\la,\sigma}=1=\delta_{\la,\sigma}$. 
This tells us that the action of $\sigma$ on $\chi$ can  be described by the action of $\sigma$ on  $\lambda$, or equivalently, on $s$ (as define in Section \ref{sec:extmaps}).  But the same argument as in \prettyref{lem:firstextmapA} 
shows that the extension map for $T_1\lhd N_1$ is $\galh$-equivariant when considering those $\la_1$ corresponding to $s$ with $\mathcal{E}(\bG^{vF}, s)\cap\irr_{2'}(\bG^{vF})\neq \emptyset$, and $\eta_1^\sigma=\eta_1$ for the same reasons as above.  Hence the action of $\sigma$ on $\Omega(\chi):=\Ind_{(N_1)_{\la_1}}^{N_1}(\Lambda_1(\la_1)\eta_1)$ is also described by its action on $s$, completing the proof. 
\end{proof}

Theorem \ref{thm:equivbij} completes \prettyref{cond:condition} from Section \ref{sec:iMN} for groups of Lie type for the prime $\ell=2$. 
In the next sections, we will deal with Condition \eqref{eq:extpart} from Section \ref{sec:iMN}. 
This was addressed for nontwisted groups whose graph automorphisms are trivial in \cite[Theorem A]{RSF22}.  

\subsection{An $\galh$-equivariant map for $\wt{G}$}\label{sec:wtGmap}

We end this section with the corresponding statement to Condition \ref{cond:condition} at the level of $\wt{G}$.  Let $G=\bG^F$ be such that $\bG$ is simple of simply connected type $\type{A}_{n-1}$ with $n\geq 3$, $\type{D}_n$ with $n\geq 4$, or $\type{E}_6$ and let $Q\in\syl_2(G)$. Write $\galh:=\galh_2$ and let $d$ be the order of $q$ modulo $4$.

From Theorem \ref{thm:equivbij} and \cite[Theorem A]{SF22}, we have a $\Gamma\times \galh$-equivariant bijection $\Omega\colon\irr_{2'}(G)\rightarrow\irr_{2'}(M)$, where $M=\norm_{\bG}(\bS)^{F}$ for a Sylow $d$-torus $\bS$ of $(\bG, F)$, which is the group $N$ in the case $q\equiv 1\mod 4$ and can be identified in $\bG^{vF}\cong G$ with the group $N_1$ in the case $q\equiv 3\mod 4$, where $N$ and $N_1$ are as in the proof of Theorem \ref{thm:equivbij} (or \cite[Sections 6, 8]{SF22}).
From \cite[Theorem 6.3]{MS16} and \cite[Theorem 6.1]{CS16}, we have  a $(\wt{G}D)_{\bS}$-equivariant bijection 
\begin{equation}\label{eqn:wtOmega}
\wt{\Omega}\colon \irr(\wt{G}|\irr_{2'}(G))\rightarrow \irr(\wt{M}|\irr_{2'}(M)),
\end{equation}
where $\wt{M}:=M \norm_{\wt{G}}(\bS)$, such that $\wt\chi$ and $\wt\Omega(\wt{\chi})$ always lie over the same character of $\zen(\wt{G})$; $\wt\Omega(\wt\chi)\beta=\wt\Omega(\wt{\chi}\beta)$ for each $\beta\in\irr(\wt{G}/G)\cong\irr(\wt{M}/M)$; and such that $\wt{\chi}$ lies over $\chi\in\irr_{2'}(G)$ if and only if $\wt\Omega(\wt{\chi})$ lies over $\Omega(\chi)\in\irr_{2'}(M)$ (see \cite[Remark 6.5]{CS16}). We next prove that $\wt\Omega$ is further $\galh$-equivariant.

\begin{proposition}\label{prop:wtOmegaequiv}
Let $G$ be one of the quasisimple groups of Lie type coming from a simply connected group, such that $G/\zent(G)$ as in Theorem \ref{thm:main}(\ref{mainlie}).  Then the map $\wt\Omega$ from \eqref{eqn:wtOmega} is $\galh_2$-equivariant.
\end{proposition}
\begin{proof}
Let $\chi=\HC_T^G(\la)_\eta\in\irr_{2'}(G)\cap\mathcal{E}(G,s)$.  Let $\wt{T}:=\wt\bT^F$ for $\wt\bT:=\bT\zent(\wt{\bG})$ and $\wt{N}:=\norm_{\wt G}(\bT)$. Then we may find $\wt{\la}\in\irr(\wt{T}|\la)$ and $\wt{\eta}\in \irr(\wt{N}_{\wt{\la}}/\wt{T})$ such that 
$\wt{\chi}:=\HC_{\wt{T}}^{\wt{G}}(\wt{\la})_{\wt\eta}\in\mathcal{E}(\wt{G}, \wt{s})$ lies above $\chi$.  The other members of $\irr(\wt{G}|\chi)$ are then of the form $\wt{\chi}\beta$ for $\beta\in\irr(\wt{G}/G)$, since $\wt{G}/G$ is abelian, by Clifford theory and that restriction from $\wt{G}$ to $G$ is multiplicity-free (see \cite[Theorem 1.7.15]{GM20}). (Recall then that $\beta=\hat z$ for some $z\in\zent(\wt{G}^\ast)$, and $\wt{\chi}\beta=\HC_{\wt{T}}^{\wt{G}}(\wt{\la}\hat z)_{\wt\eta}\in\mathcal{E}(\wt{G}, \wt{s}z)$ and $\wt\Omega(\wt{\chi}\beta)=\wt{\Omega}(\wt{\chi})\beta$.)

Let $\sigma\in\galh:=\galh_2$. Since $\wt{N}_{\wt{\la}}/\wt{T}$ is a Weyl group containing no component of type $\type{G}_2$, $\type{E}_7$, or $\type{E}_8$, we have $\wt{\eta}^\sigma=\wt\eta=\wt\eta^{(\sigma)}$ (see e.g. \cite[Proposition 5.4]{SF22}). Since the other characters appearing in the description of the action of $\sigma$ in \prettyref{thm:GaloisAct} are necessarily trivial since $\cent_{\wt{\bG}}(\wt{s})$ is connected, we see that $\wt{\chi}^\sigma=\HC_{\wt{T}}^{\wt{G}}(\wt{\la}^\sigma)_{\wt\eta}$, so that the action of $\sigma$ on $\wt\chi$ can be described by its action on $\wt{\la}$, or $\wt{s}$. (Alternatively, one sees this using \cite{SrinivasanVinroot} and the fact that unipotent characters of odd degree are rational, see \cite[Proposition 4.4]{SFgaloisHC}.) We aim to show that the same is true on the local side.

Recall from the proof of Theorem \ref{thm:equivbij} that the $MD$-equivariant extension map $\Lambda$ or $\Lambda_1$ with respect to $T\lhd N$ or $T_1\lhd N_1$ is $\galh$-equivariant on the characters $\odd(T)$ or $\odd(T_1)$ under consideration, in the cases that $\bG$ is of type $\type{A}$ or $\type{D}$. For $\bG$ of type $\type{E}_6$, the same is true by \cite[Lemma 4.1 and Corollary 4.4]{SF22}.
To ease notation, let $C:=\cent_{G}(\bg{S})$, which we identify with $T$, respectively $T_1$, and let $\Lambda$ denote the corresponding extension map with respect to $C\lhd M$ on $\odd(C)$. 
Let $\wt{C}:=\cent_{\wt{G}}(\bg{S})\lhd \wt{M}$. 
Then using \cite[Lemma 5.8(a)]{CS16}, we can define an extension map $\wt\Lambda$ with respect to $\wt{C}\lhd\wt{M}$ such that for $\la\in\irr(C)$ and  $\wt{\la}\in\irr(\wt{C}|\la)$, $\wt{\Lambda}(\wt{\la})$ is the unique extension of $\wt\la$ to $\wt{M}_{\wt{\la}}$ such that $\Res^{\wt{M}_{\wt{\la}}}_{M_{\wt{\la}}}(\wt\Lambda(\wt\la))=\Res^{{M}_{{\la}}}_{M_{\wt{\la}}}(\Lambda(\la))$. This is the extension map used to construct the  map $\wt\Omega$ from \eqref{eqn:wtOmega} (see \cite[Proposition 6.3]{CS16}).  Further, from this uniqueness and the $\galh$-equivariance of $\Lambda$, we see that $\wt\Lambda$ is also $\galh$-equivariant on characters of $\wt{C}$ above $\odd(C)$.

 Let $\wt\psi:=\wt\Omega(\wt\chi)\in\irr(\wt M|\Omega(\chi))$ and recall that $\psi:=\Omega(\chi)$ is of the form $\Ind_{{M}_{{\la_1}}}^{{M}}(\Lambda(\la_1)\eta_1)$ with $\la_1\in\irr(C)$ corresponding to $s$ and $\eta_1\in\irr(M_{\la_1}/C)$.  Following the construction in \cite[Section 6]{CS16} (which is then used in \cite[Theorem 6.3]{MS16}), we may write   
 $\wt\psi:=\Ind_{\wt{M}_{\wt{\la}_1}}^{\wt{M}}(\wt\Lambda(\wt\la_1)\wt\eta_1)$ for $\wt\la_1\in\irr(\wt{C}|\la_1)$ corresponding to $\wt{s}$ and $\wt\eta_1\in \irr(\wt M_{\wt\la_1}/\wt C)$. For the same reason as before, we have $\wt\eta_1^\sigma=\wt\eta_1$, so since $\wt\Lambda(\wt\la_1)^\sigma=\wt\Lambda(\wt\la_1^
 \sigma)$, we have  $\wt\psi^\sigma=\Ind_{\wt{M}_{\wt{\la}_1^s}}^{\wt{M}}(\wt\Lambda(\wt\la_1^\sigma)\wt\eta_1)$, so that the action of $\sigma$ on $\wt\psi$ can again be described by the action on $\wt{s}$, completing the proof.
\end{proof}

\section{Extensions and the Proof of Theorem \ref{thm:main}}\label{sec:extensions}

Keep the notation from Section \ref{sec:wtGmap}. We now wish to complete the proof of Theorem \ref{thm:main}. Thanks to Section \ref{sec:sporalt}, we may assume that $G=\bG^F$ is a group of Lie type such that $\bG$ is simple of simply connected type, $G$ is the  universal covering group of $S=G/\zent(G)$, and that $S$ is one of the simple groups listed in  Theorem \ref{thm:main}(\ref{mainlie}). We keep the notation from the previous section.

Given the $\Gamma\times \galh$-equivariant bijection $\Omega\colon\irr_{2'}(G)\rightarrow\irr_{2'}(M)$ for $G$ from Theorem \ref{thm:equivbij} and \cite[Theorem A]{SF22},
it  suffices to show that 
\[(G\rtimes \Gamma_{\chi^{\galh}}, G, \chi)_{\galh} \geqslant_c (M\rtimes \Gamma_{\chi^{\galh}}, M, \Omega(\chi))_{\galh}\] for all $\chi\in\irr_{2'}(G)$ (see Definition \ref{def:triples}).   Let $\wt{M}:=M \norm_{\wt{G}}(\bS)$.  Since $\wt{G} D$ induces all automorphisms of $G$, and $M$ can be chosen to be $D$-stable, \cite[Theorem 2.9]{NavarroSpathVallejo}  implies that it suffices to show
\[((\wt{G} D)_{ \chi^{\galh}}, G, \chi)_{\galh} \geqslant_c ((\wt{M}D)_{ \chi^{\galh}}, M, \Omega(\chi))_{\galh}\] for all $\chi\in\irr_{2'}(G)$.  
Further, \cite[Lemma 2.2]{NavarroSpathVallejo} yields that it suffices to show \[((\wt{G} D)_{ {\chi_0}^{\galh}}, G, \chi_0)_{\galh} \geqslant_c ((\wt{M}D)_{ \chi_0^{\galh}}, M, \Omega(\chi_0))_{\galh}\] for some $\wt{G}$-conjugate $\chi_0$ of each $\chi\in\irr_{2'}(G)$.  
In particular, we will work with such a $\chi_0$ satisfying that $(\wt{G}D)_{\chi_0}=\wt{G}_{\chi_0}D_{\chi_0}$ and that $\chi_0$ extends to $GD_{\chi_0}$, which exists by \cite[Theorem 4.1]{CS16} and \cite[Theorem 6.4 and Proof of Theorem 1]{MS16}. 
Using \cite[Theorems 3.1 and 3.18]{MS16} and by considering the proof of \cite[Corollary 5.3]{MS16} (see also \cite[Section 5]{CS16} for the case when $\bG$ is of type $\type{A}$), 
we also have $(\wt{M}D)_{\Omega(\chi_0)}=\wt{M}_{\Omega(\chi_0)}D_{\Omega(\chi_0)}$ and $\Omega(\chi_0)$ extends to $MD_{\Omega(\chi_0)}$. 

Note that parts (i) and (ii) of Definition \ref{def:triples} hold for these groups by construction and by the equivariance of $\Omega$. Further, part (iii)  holds by \cite[Lemmas 2.11 and 2.13]{spath12} and their proofs, once we prove (iv), using the constructions in loc. cit. and using \cite[Lemma 4.2]{RSF22}.  More specifically, choosing extensions $\tilde{\chi}$ and $\hat{\chi}$ of $\chi$ to $\tilde{G}_{\chi}$ and to $GD_\chi$, respectively, determines a projective representation $\mathcal{P}$   by \cite[Lemma 2.11]{spath12}, and similarly, extensions $\tilde{\psi}$ and $\hat{\psi}$ of $\psi$ give a projective representation $\mathcal{P}'$. If $\wt\chi$ and $\wt\psi$ have the same values on $\zent(\tilde{G})$, then the second part of part (iii) of Definition \ref{def:triples} is satisfied by the proof of \cite[Lemma 2.13]{spath12}. Further, if the extensions are constructed in a way to satisfy (iv), then \cite[Lemma 2.11]{spath12} yields that the first part of (iii) is also satisfied. Hence, it suffices to construct these extensions in a way that (iv) is satisfied.

So we see that it now suffices to prove part (iv) of Definition \ref{def:triples} for $\chi_0$. Note that by \cite[Lemma 1.8]{NavarroSpathVallejo}, it then suffices to prove part (iv) for a complete set of coset representatives for $(\wt{M}D)_{\chi_0}$ in $(\wt{M}D\times \galh)_{\chi_0}$. In particular, from the considerations in \prettyref{rem:stabs}, we see that it suffices to show (iv) for elements of $(D\galh)_{\chi_0}$.

From now on, we will write $\psi_0:=\Omega(\chi_0)$. 
Let $D_1:=O_{2'}(D_{\chi_0})$ and $D_2\in\syl_2(D_{\chi_0})$. Note  that  if $G\neq\type{D}_4(q)$, we have $D_{\chi_0}=D_1\times D_2$ is abelian.

\begin{proposition}\label{prop:ext2pt}
Let $G$ be as above with $G/\zent(G)$ as in Theorem \ref{thm:main}(\ref{mainlie}), let $\chi_0\in\irr_{2'}(G)$ be chosen as above and write $\psi_0:=\Omega(\chi_0)$. Then there exist $(\mathrm{N}_{D}(D_2) \times \galh)_{\chi_0}$-invariant extensions $\chi_2$ and $\psi_2$ of $\chi_0$ and $\psi_0$ to $GD_2$ and $MD_2$, respectively.
\end{proposition}
\begin{proof}

Let $\mu$ be the unique linear character extending  $\det\psi_0$ to $MD_2$ such that $\mu$ is trivial on $D_2$, which exists since the product is semidirect and $\det\psi_0$ is linear. Then $\mu^a=\mu$ for $a \in (\mathrm{N}_{D}(D_2) \times \galh)_{\chi_0}$, and by \cite[Lemma 6.24]{isaacs}, we have a unique extension $\psi_2$ of $\psi_0$ to $MD_2$ such that $\det\psi_2=\mu$.  Then $\psi_2^a$ is another extension, with $\det\psi_2^a=\mu^a=\mu$, so $\psi_2^a=\psi_2$. Arguing similarly, we obtain an extension $\chi_2$ of $\chi_0$ to $GD_2$ with $\chi_2^a=\chi_2$.
\end{proof}

The next proposition, found in  \cite[Proposition 8.7]{birtethesis}, is a generalized version of Proposition 2.6 and Corollary 2.7 from \cite{RSF22} (minor corrections at arXiv:2106.14745v2). 

\begin{proposition}\label{prop:2.62.7RSF}
Suppose that $\bG$ is a connected reductive group and $F$ is a Frobenius endomorphism such that $F=F_0^k\rho$ for some Frobenius endomorphism $F_0\colon\bG\rightarrow\bG$, graph automorphism $\rho\colon \bG\rightarrow\bG$ commuting with $F_0$, and positive integer $k$. 
Let $\bL\leq \bg{P}$ be $\langle F_0, F\rangle$-stable Levi and parabolic subgroups of $\bG$ with $(\norm_{G}(\bL)/L)^{F_0}=\norm_G(\bL)/L$, where $L:=\bL^F$. 
 Suppose that $\la \in \Irr(L)$ is an $F_0$-invariant cuspidal character that extends to  $\irr(\norm_{G}(\bL)_{\la} \langle F_0 \rangle)$, and let $\chi \in \Irr(G)$ be a character in the Harish--Chandra series of $\la$. Then:
\begin{enumerate}
\item For $\hat \chi \in \Irr( G\langle F_0\rangle \mid \chi)$, there exists a unique $\hat \la \in \Irr( L\langle F_0\rangle \mid \la)$ such that $\langle \HC_{ L\langle F_0\rangle}^{ G\langle F_0\rangle} \hat \la, \hat \chi \rangle \neq 0$. (Here we define $\HC_{L\langle F_0\rangle}^{G\langle F_0\rangle}:=\Ind_{\bg{P}^F\langle F_0\rangle}^{G\langle F_0\rangle} \circ \mathrm{Infl}_{L\langle F_0\rangle}^{\bg{P}^F\langle F_0\rangle}$.)
 
\item {If $\chi$ is $F_0$-invariant, this yields a bijection $\irr({G}\langle F_0\rangle\mid\chi)\rightarrow  \irr({L}\langle F_0\rangle\mid \la)$.}
\item Assume that $\chi$ is $F_0$-invariant and that  $\alpha\sigma \in \aut(G)\times \gal$ stabilizes $\chi$. Then we have $\hat \chi^{\alpha\sigma}=\hat \chi \beta$ for some $\beta \in \Irr(G \langle F_0 \rangle /G)$.  {In particular,  $\beta$ is such that $\la^{\alpha\sigma x}=\la$ and $\wh{\la}^{\alpha\sigma x}=\wh{\la}\beta$ for some $x \in \norm_{G}(\bL)$.}
\end{enumerate}
\end{proposition}

For extensions to $GD_1$ in the case of $\type{D}_4$, the following will be useful,  from \cite[Corollary 2.2]{navarrotiep2008Rational}.

\begin{lemma}\label{lem:rationalabove}
Let $X\lhd Y$ be finite groups with $|Y/X|$ odd, and let $\theta\in\irr(X)$ be rational-valued.  Then there exists a unique rational-valued character $\hat\theta\in\irr(Y|\theta)$.
\end{lemma}

\begin{proposition}\label{prop:extfield}
Let $G$ be as above with $G/\zent(G)$ as in Theorem \ref{thm:main}(\ref{mainlie}), let $\chi_0\in\irr_{2'}(G)$ be chosen as above and write $\psi_0:=\Omega(\chi_0)$. Then there exist $(D\galh)_{\chi_0}$-invariant extensions $\chi_1$ and $\psi_1$ of $\chi_0$ and $\psi_0$ to $GD_1$ and $MD_1$, respectively.
\end{proposition}
\begin{proof}
First, suppose that $\bG$ is of type $\type{D}_n$ with $n\geq 4$. Then by  \prettyref{cor:Doddcharsrational} and the proof of Theorem \ref{thm:equivbij}, 
$\chi_0$ and $\psi_0$ are rational-valued. Then by \prettyref{lem:rationalabove}, there exists a unique rational character ${\chi}_1$, resp. ${\psi}_1$, in $\irr(GD_1|\chi_0)$, resp. $\irr(MD_1|\psi_0)$.  Since $D_1$ is abelian and the characters $\chi_0$ and $\psi_0$ extend to the respective groups, it follows that $\chi_1$ and $\psi_1$ must be extensions. Then it now suffices to note that for $a\in (D\galh)_{\chi_0}$, the character $\chi_1^a$ is another rational extension of $\chi_0=\chi_0^a$, and hence $\chi_1^a=\chi_1$, and similarly $\psi_1^a=\psi_1$.

Now suppose that $\bG$ is not of type $\type{D}_n$ (in fact, the following argument also works for $\type{D}_n$ with $n\neq 4$) and let $\chi_0$ be an irreducible constituent of $\HC_T^G(\la)$. Then there exists a Frobenius map $F_0$ such that $GD_1=G\langle F_0\rangle$, $F_0^r=F$ for some positive integer $r$, and $\bg{W}^{F_0}=\bg{W}^{F}$, where $\bg{W}$ is the Weyl group of $\bG$. Further, recall that $\la\in\odd(T)$ (and, in the case $d=2$, $\la_1\in\odd(T_1)$) and can be chosen to be invariant under $F_0$ and $a\in (D\galh)_{\chi_0}$ by Lemma \ref{stab} and its proof. In particular, the hypotheses of Proposition \ref{prop:2.62.7RSF} hold in this case, as do the hypotheses in \cite[Lemmas 3.1(c), 3.3]{RSF22}  with $\tau\in\gal$ replaced with $a$. Then the exact same arguments as in loc. cit. apply here and give the desired extensions.   
\end{proof}

\begin{corollary}\label{cor:extGD}
Let $G$ be as above with $G/\zent(G)$ as in Theorem \ref{thm:main}(\ref{mainlie}), let $\chi_0\in\irr_{2'}(G)$ be chosen as above and write $\psi_0:=\Omega(\chi_0)$. Then there are $(D\galh)_{\chi_0}$-invariant extensions $\hat\chi, \hat\psi$ of $\chi_0$ and $\psi_0$ to  $GD_{\chi_0}$ and $MD_{\psi_0}$, respectively.
\end{corollary}
\begin{proof}
If $G/\zent(G)$ is not $\type{D}_{4}(q)$, this follows from Propositions \ref{prop:ext2pt} and \ref{prop:extfield} by considering the unique common extensions of ${\chi}_1$ and $\chi_2$, respectively ${\psi}_1$ and ${\psi}_2$,  to $GD_{\chi_0}$, respectively $MD_{\chi_0}$, which exist since $D$ is abelian.

So now assume that $G/\zent(G)=\type{D}_4(q)$. As before, {by Lemma \ref{lem:rationalabove} we can extend $\chi_0$ to a unique rational valued character $\chi_1 \in \Irr(G D_1)$. Moreover, we let $\chi_2 \in \Irr(GD_2)$ be the $\mathcal{H}$-stable extension as in Proposition \ref{prop:ext2pt}. The properties of $\chi_1$ imply that $\chi_1$ is $D_2$-stable. As in the proof of \cite[Corollary 6.10]{ruhstorfer}, we conclude that there exists an extension $\hat{\chi} \in \Irr(G D_\chi)$ which extends both $\chi_1$ and $\chi_2$. Now by Remark \ref{rem:stabs}, we have $(\mathcal{H} D)_\chi=\mathcal{H} D_{\chi_0}$. For $a \in \mathcal{H}$ there exists a linear character $\beta \in \Irr(D_{\chi_0})$ such that $\hat{\chi}^a=\hat{\chi} \beta$. Arguing as in \cite[Corollary 6.10]{ruhstorfer} we observe that $\beta_{D_{\chi_0}}=1$ and so $\hat{\chi}$ is $\mathcal{H}$-stable.} Replacing $G$ by $M$, we may argue the same to obtain the extension $\hat\psi$.
\end{proof}

\begin{corollary}\label{cor:extwtG}
Let $G$ be as above with $G/\zent(G)$ as in Theorem \ref{thm:main}(\ref{mainlie}), let $\chi_0\in\irr_{2'}(G)$ be chosen as above and write $\psi_0:=\Omega(\chi_0)$. Let $a\in (D\galh)_{\chi_0}$.   Then there are extensions $\wt\chi, \wt\psi$ of $\chi_0$ and $\psi_0$ to  $\wt{G}_{\chi_0}$ and $\wt{M}_{\psi_0}$, respectively, satisfying that for $a\in (D\galh)_{\chi_0}$, if $\wt\chi^a=\wt\chi\mu_a$ and $\wt\psi^a=\wt\psi\mu_a'$ for $\mu_a, \mu_a'\in \irr(\wt{G}_{\chi_0}/G)\cong\irr(\wt{M}_{\chi_0}/M)$, then $\mu_a=\mu_a'$.

Hence, 
Definition \ref{def:triples}(iv) 
holds for $\chi_0$ and for any $a \in(D\galh)_{\chi_0}$. In particular, $G$ satisfies Condition \eqref{eq:extpart} from Section \ref{sec:iMN}.
\end{corollary}
\begin{proof}
First, note that the second statement follows from the first and from Corollary \ref{cor:extGD}, by using \cite[Lemma 4.2]{RSF22} to obtain the desired extensions in Definition \ref{def:triples}(iv).

Recall that $\chi_0$ either extends to $\wt{G}$ or we have $|\zen(G)|$ and $|\wt{G}/G\zen(\wt{G})|$ are each a power of $2$ and $\chi_0$ is stabilized by $\galh$.  
In the latter case, note that $\chi_0$ and $\psi_0$ may be extended trivially to $G\zen(\wt{G})$ and $M\zen(\wt{G})$, respectively, since they are trivial on $\zen(G)$ (as they have odd degree).  So, we may identify $\chi_0$ and $\psi_0$ with these extensions.  
Then since $\chi_0$ has odd degree, the determinantal order $o(\chi_0)$ satisfies $o(\chi_0)=1$ (since $G$ is perfect), and $|\wt{G}/G\zen(\wt{G})|$ is a power of $2$, it follows that there is a unique extension $\wt\chi$ of $\chi_0$ to $\wt{G}_{\chi_0}$ with $o(\wt\chi)=1$, by \cite[Corollary 8.16]{isaacs}.  Then since $\wt\chi^a$ is another extension with $o(\wt\chi^a)=1$, we have $\wt\chi^a=\wt\chi$. 
Further, in these cases, recall that $\chi_0$ is rational-valued, so that $a\in D_{\chi_0}\galh$.  Then $\wt{\chi}_0:=\Ind_{\wt{G}_{\chi_0}}^{\wt{G}}(\wt\chi)$ is also $a$-invariant.  Then $\wt{\psi}_0:=\wt{\Omega}(\wt{\chi}_0)$ is $a$-invariant since $\wt{\Omega}$ is $(\wt{G}D\times\galh)$-equivariant by \prettyref{prop:wtOmegaequiv} and the discussion before it. 
 Then since $\wt{\psi}_0$ must lie above $\psi_0$ (by the properties of $\wt\Omega$ discussed above) and $\wt{M}/M$ is abelian, we have $\wt{\psi}_0=\Ind_{\wt{M}_{\psi_0}}^{\wt{M}}(\wt\psi)$ for some extension $\wt\psi$ of $\psi_0$ to $\wt{M}_{\psi_0}$.  
 Now, since $\psi_0^a=\psi_0$, we have $\wt\psi^a=\wt\psi\beta$ for some linear $\beta\in\irr(\wt{M}_{\psi_0}/M)$.  
 But then  $\wt{\psi}_0=\wt\psi_0^a=\Ind_{\wt{M}_{\psi_0}}^{\wt{M}}(\wt\psi^a)=\Ind_{\wt{M}_{\psi_0}}^{\wt{M}}(\wt\psi\beta)=\Ind_{\wt{M}_{\psi_0}}^{\wt{M}}(\wt\psi)\beta=\wt\psi_0\beta$, where the second-to-last equality follows from \cite[Problem (5.3)]{isaacs}.  Then $\beta=1$ and $\wt\psi^a=\wt\psi$.

So, we are left with the case that $\chi_0$ and $\psi_0$ extend to $\wt{G}$ and $\wt{M}$, respectively. In this case, let $\wt{\chi}_0$ be an extension of $\chi_0$, so that $\wt{\chi}_0^a=\wt{\chi}_0\beta$ for some $\beta\in\irr(\wt{G}/G)$ and let $\wt\psi_0:=\wt\Omega(\wt\chi_0)$. Then $\wt\psi_0$ is an extension of $\psi_0$ as well, since it lies above $\psi_0$ and $\wt{M}/M$ is abelian. Then, recalling that $\wt\Omega$ is $(\irr(\wt{G}/G)\times D\galh)$-equivariant by \prettyref{prop:wtOmegaequiv} and the properties discussed after \eqref{eqn:wtOmega}, we have $\wt\psi_0^a=\wt\Omega(\wt\chi_0)^a=\wt\Omega(\wt{\chi}_0^a)=\wt\Omega(\wt{\chi}_0\beta)=\wt\Omega(\wt\chi_0)\beta=\wt\psi_0\beta$. The result now follows, taking $\wt\chi:=\wt\chi_0, \wt\psi:=\wt\psi_0,$ and $\mu_a=\beta=\mu_a'$. The final sentence follows from the considerations made at the beginning of this section.
\end{proof}

From here, we see that we have proved the inductive McKay--Navarro conditions for the groups in Theorem \ref{thm:main}(\ref{mainlie}), which combined with \prettyref{prop:sporadicalt} completes the proof of Theorem \ref{thm:main}.  In conjunction with those results discussed in the introduction, this completes the proof of the McKay--Navarro conjecture for $\ell=2$.


\def\cprime{$'$} \def\cprime{$'$}

\end{document}